\documentclass{article}

\RequirePackage{natbib}
\RequirePackage[colorlinks,citecolor=blue,urlcolor=blue]{hyperref}

\usepackage{a4wide}

\usepackage{amsmath}
\usepackage{amsfonts}
\usepackage{amssymb}
\usepackage{amsthm}

\usepackage{graphicx}
  \setkeys{Gin}{width=1\textwidth}

\usepackage[english]{babel}
\usepackage{textcomp}
\usepackage[latin1]{inputenc}

\usepackage{dsfont}

\newcommand{\N}{\mathbb{N}}

\newcommand{\R}{\mathbb{R}}

\newcommand{\one}{\mathds{1}}

\newcommand{\esp}{\mathbb{E}}
\DeclareMathOperator{\cov}{Cov}
\DeclareMathOperator{\var}{Var}
\newcommand{\proba}{\mathbb{P}}
\newcommand{\loi}[1]{\mathcal{#1}}

\newcommand{\conv}{\xrightarrow{}}
\newcommand{\convloi}{\xrightarrow{d}}
\newcommand{\convproba}{\xrightarrow{\proba}}
\newcommand{\convps}{\xrightarrow{a.s.}}

\newcommand{\convn}{\xrightarrow[n\conv+\infty]{}}

\newcommand{\convnps}{\xrightarrow[n\conv+\infty]{a.s.}}

\newcommand{\ud}{\,\mathrm{d}}
\newcommand{\uD}{\mathrm{D}}

\newcommand{\po}{\mathrm{o}}
\newcommand{\go}{\mathrm{O}}
\newcommand{\pop}{\mathrm{o_\proba}}
\newcommand{\gop}{\mathrm{O_\proba}}

\newcommand{\interff}[2]{[#1,\,#2]}
\newcommand{\interoo}[2]{]#1,\,#2[}
\newcommand{\interof}[2]{]#1,\,#2]}
\newcommand{\interfo}[2]{[#1,\,#2[}

\newcommand{\Interoo}[2]{\left]{#1,\,#2}\right[}

\newtheorem{theorem}{Theorem}[section]

\newtheorem{corollary}[theorem]{Corollary}

\newtheorem{lemma}[theorem]{Lemma}

\newtheorem{proposition}[theorem]{Proposition}

\makeatletter

\@addtoreset{equation}{section}
\makeatother

\begin{document}

\title{Consistency of the posterior distribution and MLE for piecewise linear regression}

\author{Tristan  Launay${}^{1,2}$ \and Anne Philippe${}^{1}$  \and  Sophie Lamarche${}^{2}$ }

\date{\small ${}^{1}$ Laboratoire de Math\'ematiques Jean Leray, \\2 Rue
  de la Houssinière -- BP 92208, 44322 Nantes Cedex 3, France  
\\ 
${}^{2}$  Electricit\'e de France R\&D, 1 Avenue du G\'en\'eral de
Gaulle, \\ 92141 Clamart Cedex, France}

\maketitle

\begin{abstract}
We prove the weak consistency of the posterior distribution and that of the Bayes estimator for a two-phase piecewise linear regression mdoel where the break-point is unknown. The non-differentiability of the likelihood of the model with regard to the break-point parameter induces technical difficulties that we overcome by creating a regularised version of the problem at hand. We first recover the strong consistency of the quantities of interest for the regularised version, using results about the MLE, and we then prove that the regularised version and the original version of the problem share the same asymptotic properties.

\noindent \textit{keywords : }  consistency ; asymptotic distribution
; posterior distribution ; MLE ; piecewise regression. 
\end{abstract}

\section{Introduction}\label{sec:intro}
We consider a continuous segmented regression model with 2 phases, one of them (the rightmost) being zero. Let \(u\) be the unknown breakpoint and \(\gamma\in\R\) be the unknown regression coefficient of the non zero phase. The observations \(X_{1:n} = (X_1,\ldots,X_n)\) depend on an exogenous variable that we denote \(t_{1:n} = (t_1, \ldots, t_n)\) via the model given for \(i=1,\ldots,n\) by
\begin{align}\label{eq.modelesylwester}
X_i &= \mu(\eta, t_i) + \xi_i:= \gamma\cdot (t_i - u)\one_{\interfo{t_i}{+\infty}}(u) + \xi_i,
\end{align}
where \((\xi_i)_{i\in\N}\) is a sequence of independent and identically distributed (i.i.d.) random variables with a common centered Gaussian distribution of unknown variance \(\sigma^2\), \(\loi{N}(0,\sigma^2)\), and where \(\one_A\) denotes the indicator function of a set \(A\).

Such a model is for instance used in practice to estimate and predict the heating part of the electricity demand in France. See \cite{Bruhns} for the definition of the complete model and \cite{Launay1} for a Bayesian approach. In this particular case, \(u\) corresponds to the heating threshold above which the temperatures \(t_{1:n}\) do not have any effect over the electricity load, and \(\gamma\) corresponds to the heating gradient i.e. the strength of the described heating effect.

The work presented in this paper is most notably inspired by the results developed in \cite{Ghosh} and \cite{Feder}.

\citeauthor{Feder} proved the weak consistency of the least squares estimator in segmented regression problems with a known finite number of phases under the hypotheses of his Theorem 3.10 and some additional assumptions disseminated throughout his paper, amongst which we find that the empirical cumulative distribution functions of the temperatures at the \(n\)-th step \(t_{n1}, \ldots, t_{nn}\) are required to converge to a cumulative distribution function, say \(F_n\) converges to \(F\), which is of course to be compared to our own Assumption (A1). \citeauthor{Feder} also derived the asymptotic distribution of the least squares estimator under the same set of assumptions. Unfortunately there are a few typographical errors in his paper (most notably resulting in the disappearance of \(\sigma_0^2\) from the asymptotic variance matrix in his main theorems), and he also did not include \(\widehat{\sigma}_n^2\) in his study of the asymptotic distribution.

The asymptotic behaviour of the posterior distribution is a central question that has already been raised in the past. For example, \citeauthor{Ghosh} worked out the limit of the posterior distribution in a general and regular enough i.i.d. setup. In particular they manage to derive the asymptotic normality of the posterior distribution under third-order differentiability conditions. There are also a number of works dealing with some kind of non regularity, like these of \cite{Sareen} which consider data the support of which depends on the parameters to be estimated, or those of \cite{Ibragimov} which offer the limiting behaviour of the likelihood ratio for a wide range of i.i.d. models whose likelihood may present different types of singularity. Unfortunately, the heating part model presented here does not fall into any of these already studied categories.

In this paper, we show that the results of \citeauthor{Ghosh} can be extended to a non i.i.d. two-phase regression model. We do so by using the original idea found in \cite{Sylwester}\footnote{\citeauthor{Sylwester} indeed considers the same model as we do here, however his asymptotic results are false due to an incorrect reparametrisation of the problem and an error in the proof of his Theorem 3.5.}: we introduce a new, regularised version of the problem called pseudo-problem, later reprised by \citeauthor{Feder}. The pseudo-problem consists in removing a fraction of the observations in the neighbourhood of the true parameter to obtain a differentiable likelihood function. We first recover the results of \citeauthor{Ghosh} for this pseudo-problem and then extend these results to the (full) problem by showing that the estimates for the problem and the pseudo-problem have the same asymptotic behaviour.

From this point on, we shall denote the parameters \(\theta = (\gamma, u, \sigma^2) = (\eta, \sigma^2)\) and \(\theta_0\) will denote the true value of \(\theta\). We may also occasionally refer to the intercept of the model as \(\beta = -\gamma u\). The log-likelihood of the \(n\) first observations \(X_{1:n}\) of the model will be denoted
\begin{align}\label{eq.defloglikelihood1n}
l_{1:n}(X_{1:n}|\theta) &= \sum_{i=1}^n l_i(X_i|\theta) \\
&=-\frac{n}{2}\log\left(2\pi\sigma^2\right) - \sum_{i=1}^n  \frac{1}{2\sigma^2}\left(X_i-\gamma\cdot (t_i - u)\one_{\interfo{t_i}{+\infty}}(u)\right)^2,
\end{align}
where \(l_i(X_i|\theta)\) designates the log-likelihood of the \(i\)-th observation \(X_i\), i.e.
\begin{align}\label{eq.defloglikelihoodi}
l_i(X_{1:n}|\theta) &= -\frac{1}{2}\log\left(2\pi\sigma^2\right) - \frac{1}{2\sigma^2}\left(X_i-\gamma\cdot (t_i - u)\one_{\interfo{t_i}{+\infty}}(u)\right)^2.
\end{align}
Notice that we do not mention explicitly the link between the likelihood \(l\) and the sequence of temperatures \((t_n)_{n\in\N}\) in these notations, so as to keep them as minimal as possible. The least square estimator \({\widehat\theta}_n\) of \(\theta\) being also the maximum likelihood estimator of the model, we refer to it as the MLE.

Throughout the rest of this paper we work under the following assumptions
\vspace{0.25em}\par\noindent\textsc{Assumption (A1).~} The sequence of temperatures (exogenous variable) \((t_{n})_{n\in\N}\) belongs to a compact set \(\interff{\underline{u}}{\overline{u}}\) and the sequence of the empirical cumulative distribution functions \((F_n)_{n\in\N}\) of \((t_1, \ldots, t_n)\), defined by \[F_n(u) = \frac{1}{n}\sum_{i=1}^n \one_{\interfo{t_i}{+\infty}}(u),\] converges pointwise to a function \(F\) where \(F\) is a cumulative distribution function itself, which is continuously differentiable over \(\interff{\underline{u}}{\overline{u}}\).
\vspace{0.25em}\par\noindent\textsc{Remark 1.~} Due to a counterpart to Dini's Theorem \cite[see Theorem \ref{theo.Polya} taken from][(p81)]{Polya}, \(F_n\) converges to \(F\) uniformly over \(\interff{\underline{u}}{\overline{u}}\).
\vspace{0.25em}\par\noindent\textsc{Remark 2.~} Let \(h\) be a continuous, bounded function on \(\interff{\underline{u}}{\overline{u}}\). As an immediate consequence of this assumption, for any interval \(I\subset\interff{\underline{u}}{\overline{u}}\), we have, as \(n\conv+\infty\)
\begin{align*}
\frac{1}{n} \sum_{i=1}^n h(t_i) \one_{I}(t_i) &= \int_{I} h(t) \ud F_n(t) \conv \int_{I} h(t) \ud F(t) = \int_{I} h(t) f(t) \ud t,
\end{align*}
the convergence holding true by definition of the convergence of probability measures \cite[see][pages 14--16]{BillingsleyCPM}.
In particular, for \(I=\interff{\underline{u}}{\overline{u}}\) and \(I=\interof{-\infty}{u}\) we get, as \(n\conv+\infty\)
\begin{align*}
\frac{1}{n} \sum_{i=1}^n h(t_i) &\conv \int_{\underline{u}}^{\overline{u}} h(t) f(t) \ud t, &
\frac{1}{n} \sum_{i=1}^n h(t_i) \one_{\interfo{t_i}{+\infty}}(u) &\conv \int_{\underline{u}}^{u} h(t) f(t) \ud t.
\end{align*}
\vspace{0.25em}\par\noindent\textsc{Remark 3.~} It is a general enough assumption which encompasses both the common cases of i.i.d. continuous random variables and periodic (non random) variables under a continous (e.g. Gaussian) noise.

\vspace{0.25em}\par\noindent\textsc{Assumption (A2).~}  \(\theta_0\in\Theta\), where the parameter space \(\Theta\) is defined (for identifiability) as \[\Theta = \R^*\times\interoo{\underline{u}}{\overline{u}}\times\R_+^*,\] where \(\R^* = \{ x \in \R\;, x \neq 0 \}\) and \(\R_+^* = \{ x \in \R\;, x > 0 \}\).
\vspace{0.25em}\par\noindent\textsc{Assumption (A3).~}  \(f=F^\prime\) does not vanish (i.e. is positive) on \(\interoo{\underline{u}}{\overline{u}}\).
\vspace{0.25em}\par\noindent\textsc{Assumption (A4).~}  There exists \(K\subset\Theta\) a compact subset of the parameter space \(\Theta\) such that \({\widehat\theta}_n \in K\) for any \(n\) large enough.

\vspace{0.25em}\par The paper is organised as follows. In Section \ref{sec:bayescons}, we present the Bayesian consistency (the proofs involved there rely on the asymptotic distribution of the MLE) and introduce the concept of pseudo-problem. In Section \ref{sec:strongconsMLE}, we prove that the MLE for the full problem is strongly consistent. In Section \ref{sec:asympdistMLE} we derive the asymptotic distribution of the MLE using the results of Section \ref{sec:strongconsMLE}: to do so, we first derive the asymptotic distribution of the MLE for the pseudo-problem and then show that the MLEs for the pseudo-problem and the problem share the same asymptotic distribution. We discuss these results in Section \ref{sec:discussion}. The extensive proofs of the main results are found in Section \ref{sec:proofmainresults} while the most technical results are pushed back into Section \ref{sec:prooftechnical} at the end of this paper.

\vspace{0.25em}\par\noindent\textsc{Notations.~}
Whenever mentioned, the \(\go\) and \(\po\) notations will be used to designate a.s. \(\go\) and a.s. \(\po\) respectively, unless there are indexed with \(\proba\) as in \(\gop\) and \(\pop\), in which case they will designate \(\go\) and \(\po\) in probability respectively.

Hereafter we will use the notation \(A^c\) for the complement of the set \(A\) and \(B(x,r)\) for the open ball of radius \(r\) centred at \(x\) i.e. \(B(x, r) = \{x^\prime,\; \|x^\prime-x\| < r\}\).

\section{Bayesian consistency}\label{sec:bayescons}
In this Section, we show that the posterior distribution of \(\theta\) given \((X_1,\ldots,X_n)\) asymptotically favours any neighbourhood of \(\theta_0\) as long as the prior distribution itself charges a (possibly different) neighbourhood of \(\theta_0\) (see Theorem \ref{theo.strongconsistencyposterior}). We then present in Theorem \ref{theo.asympposteriorpb} the main result of this paper i.e. the convergence of posterior distribution with suitable normalisation to a Gaussian distribution.

\subsection{Consistency and asymptotic normality of the posterior distribution}
\begin{theorem}\label{theo.strongconsistencyposterior} Let \(\pi(\cdot)\) be a prior distribution on \(\theta\), continuous and positive on a neighbourhood of \(\theta_0\) and let \(U\) be a neighbourhood of \(\theta_0\), then under Assumptions (A1)--(A4), as \(n\conv+\infty\),
\begin{align}
\int_{U} \pi(\theta|X_{1:n}) \ud \theta \convps 1. \label{eq.strongconsistencyposterior}
\end{align}
\end{theorem}

\begin{proof}[Proof for Theorem \ref{theo.strongconsistencyposterior}]The proof is very similar to the one given in \cite{GhoshBNP} for a model with i.i.d. observations.
Let \(\delta > 0\) small enough so that \(B(\theta_0, \delta) \subset U\). Since
\begin{align*}
\int_{U} \pi(\theta|X_{1:n}) \ud \theta
&=\dfrac{1}{1 + \dfrac{\int_{U^c} \pi(\theta) \exp[l_{1:n}(X_{1:n}|\theta)-l_{1:n}(X_{1:n}|\theta_0)] \ud \theta}{\int_{U} \pi(\theta) \exp[l_{1:n}(X_{1:n}|\theta)-l_{1:n}(X_{1:n}|\theta_0)] \ud \theta} } \\
& \leqslant \dfrac{1}{1 +\dfrac{\int_{B^c(\theta_0, \delta)} \pi(\theta) \exp[l_{1:n}(X_{1:n}|\theta)-l_{1:n}(X_{1:n}|\theta_0)] \ud \theta}{\int_{B(\theta_0, \delta)} \pi(\theta) \exp[l_{1:n}(X_{1:n}|\theta)-l_{1:n}(X_{1:n}|\theta_0)] \ud \theta} }
\end{align*}
it will suffice to show that
\begin{align}
\dfrac{\int_{B^c(\theta_0, \delta)} \pi(\theta) \exp[l_{1:n}(X_{1:n}|\theta)-l_{1:n}(X_{1:n}|\theta_0)] \ud \theta}{\int_{B(\theta_0, \delta)} \pi(\theta) \exp[l_{1:n}(X_{1:n}|\theta)-l_{1:n}(X_{1:n}|\theta_0)] \ud \theta} &\convps 0. \label{eq.fractiontrongconsistencyposterior}
\end{align}

To prove \eqref{eq.fractiontrongconsistencyposterior} we adequately majorate its numerator and minorate its denominator. The majoration mainly relies on Proposition \ref{prop.majorationdiffdeL} while the minoration is derived without any major difficulties. The comprehensive proof of \eqref{eq.fractiontrongconsistencyposterior} can be found in Section \ref{subsec:proof2} on page \pageref{subsec:proof2}.
\end{proof}

Let \(\theta\in\Theta\), we now define \(I(\theta)\), the asymptotic Fisher Information matrix \(I(\theta)\) of the model, as the symmetric matrix given by
\begin{align}\label{eq.lemmaasympFishInfMatrix}
I(\theta) &=
\left[\begin{array}{ccc}
    \displaystyle \sigma^{-2} \int_{\underline{u}}^{u}(t-u)^2 \ud F(t) & \displaystyle -\sigma^{-2} \gamma \int_{\underline{u}}^{u}(t-u) \ud F(t) & 0\\
    & \displaystyle \sigma^{-2} \gamma^2 \int_{\underline{u}}^{u}1 \ud F(t) & 0 \\
    & & \displaystyle \frac{1}{2}\sigma^{-4}
    \end{array}\right].
\end{align}
It is obviously positive and definite since all its principal minor determinants are positive. The proof of the fact that it is indeed the limiting matrix of the Fisher Information matrix of the model is deferred to Lemma \ref{lemma.Bn}.

\begin{theorem}\label{theo.asympposteriorpb}
Let \(\pi(\cdot)\) be a prior distribution on \(\theta\), continuous and positive at \(\theta_0\), and let \(k_0\in\N\) such that
\begin{align*}
\int_\Theta \|\theta\|^{k_0} \pi(\theta) \ud \theta < +\infty,
\end{align*}
and denote
\begin{align}
t &= n^{\frac{1}{2}} (\theta-\widehat{\theta}_n), \label{eq.definitiont}
\end{align}
and \(\widetilde{\pi}_n(\cdot|X_{1:n})\) the posterior density of \(t\) given \(X_{1:n}\), then under Assumptions (A1)--(A4), for any \(0\leqslant k\leqslant k_0\), as \(n\conv+\infty\),
\begin{align}\label{eq.asympposterior}
\int_{\R^3} \|t\|^k \left|\widetilde{\pi}_n(t|X_{1:n}) - (2\pi)^{-\frac{3}{2}}|I(\theta_0)|^{\frac{1}{2}}e^{-\frac{1}{2}t^\prime I(\theta_0) t}\right| \ud t \convproba 0,
\end{align}
where \(I(\theta)\) is defined in \eqref{eq.lemmaasympFishInfMatrix} and \(\theta_0\) the true value of the parameter.
\end{theorem}

The proof Theorem \ref{theo.asympposteriorpb} relies on the consistency of the pseudo-problem, first introduced in \cite{Sylwester}, that we define in the next few paragraphs.

\subsection{Pseudo-problem}\label{subsec.pseudopb}
The major challenge in proving Theorem \ref{theo.asympposteriorpb} is that the typical arguments usually used to derive the asymptotic behaviour of the posterior distribution \cite[see][for example]{Ghosh} do not directly apply here. The proof provided by \citeauthor{Ghosh} requires a Taylor expansion of the likelihood of the model up to the third order at the MLE, and the likelihood of the model we consider here at the \(n\)-th step is very obviously not continuously differentiable w.r.t. \(u\) in each observed temperature \(t_i\), \(i=1,\ldots,n\). Note that the problem only grows worse as the number of observations increases.

To overcome this difficulty we follow the original idea first introduced in \cite{Sylwester}, and later used again in \cite{Feder}: we introduce a pseudo-problem for which we are able to recover the classical results and show that the differences between the estimates for the problem and the pseudo-problem are, in a sense, negligeable. The pseudo-problem is obtained by deleting all the observations within intervals \(D_n\) of respective sizes \(d_n\) centred around \(u_0\). The intervals \(D_n\) are defined as
\begin{align*}
D_n &= \Interoo{u_0-\frac{d_n}{2}}{u_0+\frac{d_n}{2}},
\end{align*}
and their sizes \(d_n\) are chosen such that as \(n\conv +\infty\)
\begin{align}
d_n &\conv 0, &n^{-\frac{1}{2}} (\log n) \cdot d_n^{-1} &\conv 0. \label{eq.definitiondn}
\end{align}
This new problem is called pseudo-problem because the value of \(u_0\) is unknown and we therefore cannot in practice delete these observations. Note that the actual choice of the sequence \((d_n)_{n\in\N}\) does not influence the rest of the results in any way, as long as it satisfies to conditions \eqref{eq.definitiondn}. It thus does not matter at all whether one chooses (for instance) \(d_n = n^{-\frac{1}{4}}\) or \(d_n = \log^{-1} n\).

Let us denote \(n^{**}\) the number of observations deleted from the original problem, and \(n^*=n-n^{**}\) the sample size of the pseudo-problem. Generally speaking, quantities annotated with a single asterisk \(^*\) will refer to the pseudo-problem. \(l_{1:n}^*(X_{1:n}|\theta)\) will thus designate the likelihood of the pseudo-problem i.e. (reindexing observations whenever necessary)
\begin{align}
l_{1:n}^*(X_{1:n}|\theta) &= -\frac{n^*}{2}\log\left(2\pi\sigma^2\right) - \sum_{i=1}^{n^*}  \frac{1}{2\sigma^2}\left(X_i-\gamma\cdot (t_i - u)\one_{\interfo{t_i}{+\infty}}(u)\right)^2.
\end{align}

On one hand, from an asymptotic point of view, the removal of those \(n^{**}\) observations should not have any kind of impact on the distribution theory. The intuitive idea is that deleting \(n^{**}\) observations takes away only a fraction \(n^{**}/n\) of the information which asymptotically approaches zero as will be shown below. The first condition \eqref{eq.definitiondn} seems only a natural requirement if we ever hope to prove that the MLE for the problem and the pseudo-problem behave asymptotically in a similar manner (we will show they do in Theorem \ref{theo.asympdist}, see equation \eqref{eq.proximityproba}).

On the other hand, assuming the MLE is consistent (we will show it is, in Theorem \ref{theo.strongconsistencywithrate}) and assuming that the sizes \(d_n\) are carefully chosen so that the sequence \((\widehat{u}_n)_{n\in\N}\) falls into the designed sequence of intervals \((D_n)_{n\in\N}\) (see Proposition \ref{prop.asympdistpseudo}, whose proof the second condition \eqref{eq.definitiondn} is tailored for), these regions will provide open neighbourhoods of the MLE over which the likelihood of the pseudo-problem will be differentiable. The pseudo-problem can therefore be thought of as a locally regularised version of the problem (locally because we are only interested in the differentiability of the likelihood over a neighbourhood of the MLE). We should thus be able to retrieve the usual results for the pseudo-problem with a bit of work. It will be shown that this is indeed the case (see Theorem \ref{theo.asympposteriorpseudo}).

If the sequence \((d_n)_{n\in\N}\) satisfies to conditions \eqref{eq.definitiondn}, then as \(n\conv+\infty\),
\begin{align*}
\frac{n^{**}}{n} &\conv 0, &\frac{n^{*}}{n} &\conv 1.
\end{align*}

Using the uniform convergence of \(F_n\) to \(F\)  over any compact subset (see Assumption (A1), and its Remark 1), we indeed find via a Taylor-Lagrange approximation
\begin{align*}
\frac{n^{**}}{n}
&= F_n\left(u_0+\frac{d_n}{2}\right) - F_n\left(u_0-\frac{d_n}{2}\right)          \\
&= F  \left(u_0+\frac{d_n}{2}\right) - F  \left(u_0-\frac{d_n}{2}\right) + \po(1) \\
&= d_n \cdot f(u_n) + \po(1),
\end{align*}
where \(u_n \in D_n\), so that in the end, since \(u_n\conv u_0\) and \(f\) is continuous and positive at \(u_0\), we have a.s.
\begin{align*}
\frac{n^{**}}{n}
&= d_n \cdot (f(u_0) + \po(1)) + \po(1) \conv 0.
\end{align*}

We now recover the asymptotic normality of the posterior distribution for the pseudo problem. 

\begin{theorem}\label{theo.asympposteriorpseudo}
Let \(\pi(\cdot)\) be a prior distribution on \(\theta\), continuous and positive at \(\theta_0\), and let \(k_0\in\N\) such that
\begin{align*}
\int_\Theta \|\theta\|^{k_0} \pi(\theta) \ud \theta < +\infty.
\end{align*}
and denote
\begin{align}
t^* &= n^{\frac{1}{2}} (\theta-\widehat{\theta}_n^*), \label{eq.definitiontstar}
\end{align}
and \(\widetilde{\pi}_n^*(\cdot|X_{1:n})\) the posterior density of \(t^*\) given \(X_{1:n}\), then under Assumptions (A1)--(A4) and conditions \eqref{eq.definitiondn}, for any \(0\leqslant k\leqslant k_0\), as \(n\conv+\infty\),
\begin{align}\label{eq.asympposteriorpseudo}
\int_{\R^3} \|t\|^{k} \left|\widetilde{\pi}_n^*(t|X_{1:n}) - (2\pi)^{-\frac{3}{2}}|I(\theta_0)|^{\frac{1}{2}}e^{-\frac{1}{2}t^\prime I(\theta_0) t}\right| \ud t \convps 0,
\end{align}
where \(I(\theta)\) is defined in \eqref{eq.lemmaasympFishInfMatrix}.
\end{theorem}

\begin{proof}[Proof of Theorem \ref{theo.asympposteriorpseudo}]
The extensive proof, to be found in Section \ref{subsec:proof2}, was inspired by that of Theorem 4.2 in \cite{Ghosh} which deals with the case where the observations \(X_1,\ldots,X_n\) are independent and identically distributed and where the (univariate) log-likelihood is differentiable in a fixed small neighbourhood of \(\theta_0\). We tweaked the original proof of \citeauthor{Ghosh} so that we could deal with independent but not identically distributed observations and a (multivariate) log-likelihood that is guaranteed differentiable only on a decreasing small neighbourhood of \(\theta_0\).
\end{proof}

%

\subsection{From the pseudo-problem to the original problem}
We now give a short proof of Theorem \ref{theo.asympposteriorpb}. As we previously announced, it relies upon its counterpart for the pseudo-problem, i.e. Theorem \ref{theo.asympposteriorpseudo}.
\begin{proof}[Proof of Theorem \ref{theo.asympposteriorpb}]
Recalling the definition of \(t\) and \(t^*\) given in \eqref{eq.definitiont} and \eqref{eq.definitiontstar} we observe that
\begin{align*}
t &= t^* + n^{\frac{1}{2}}(\widehat{\theta}_n^*-\widehat{\theta}_n).
\end{align*}
Thus the posterior distribution of \(t^*\) and that of \(t\), given \(X_{1:n}\) are linked together via
\begin{align}
\widetilde{\pi}_n(t|X_{1:n}) &= \widetilde{\pi}_n^*(t - \alpha_n|X_{1:n})\label{eq.relationpinetpinstar}
\intertext{where}
\alpha_n &= n^{\frac{1}{2}}(\widehat{\theta}_n^*-\widehat{\theta}_n). \nonumber
\end{align}
Relationship \eqref{eq.relationpinetpinstar} allows us to write
\begin{align*}
&~\int_{\R^3} \|t\|^k \left|\widetilde{\pi}_n(t|X_{1:n}) - (2\pi)^{-\frac{3}{2}}|I(\theta_0)|^{\frac{1}{2}}e^{-\frac{1}{2}t^\prime I(\theta_0) t}\right| \ud t \\
&= \int_{\R^3} \|t\|^k \left|\widetilde{\pi}_n^*(t - \alpha_n|X_{1:n}) - (2\pi)^{-\frac{3}{2}}|I(\theta_0)|^{\frac{1}{2}}e^{-\frac{1}{2}t^\prime I(\theta_0) t}\right| \ud t \\
&= \int_{\R^3} \|t+\alpha_n\|^k \left|\widetilde{\pi}_n^*(t|X_{1:n}) - (2\pi)^{-\frac{3}{2}}|I(\theta_0)|^{\frac{1}{2}}e^{-\frac{1}{2}(t+\alpha_n)^\prime I(\theta_0) (t+\alpha_n)}\right| \ud t \\
&\leqslant \int_{\R^3} \|t+\alpha_n\|^k \left|\widetilde{\pi}_n^*(t|X_{1:n}) - (2\pi)^{-\frac{3}{2}}|I(\theta_0)|^{\frac{1}{2}}e^{-\frac{1}{2}t^\prime I(\theta_0) t}\right| \ud t \\
&\qquad + (2\pi)^{-\frac{3}{2}}|I(\theta_0)|^{\frac{1}{2}} \int_{\R^3} \|t+\alpha_n\|^k \left|e^{-\frac{1}{2}(t+\alpha_n)^\prime I(\theta_0) (t+\alpha_n)} - e^{-\frac{1}{2}t^\prime I(\theta_0) t}\right| \ud t
\end{align*}
Theorem \ref{theo.asympposteriorpseudo} ensures that the first integral on the right hand side of this last inequality goes to zero in probability. It therefore suffices to show that the second integral goes to zero in probability to end the proof, i.e. that as \(n\conv+\infty\)
\begin{align}
\int_{\R^3} \|t+\alpha_n\|^k \left|e^{-\frac{1}{2}(t+\alpha_n)^\prime I(\theta_0) (t+\alpha_n)} - e^{-\frac{1}{2}t^\prime I(\theta_0) t}\right| \ud t &\convproba 0. \label{eq.diffdedeuxgaussiennes}
\end{align}
But the proof of \eqref{eq.diffdedeuxgaussiennes} is straightforward knowing that \(\alpha_n \convproba 0\) (see \eqref{eq.proximityproba}) and using dominated convergence.
\end{proof}

As an immediate consequence of Theorem \ref{theo.asympposteriorpb} we want to mention the weak consistency of the Bayes estimator.
\begin{corollary}\label{cor.bayesestpb}
Let \(\pi(\cdot)\) a prior distribution on \(\theta\), continuous and positive at \(\theta_0\), such that
\begin{align*}
\int_\Theta \|\theta\| \pi(\theta) \ud \theta < +\infty,
\end{align*}
and denote
\begin{align*}
\widetilde{\theta}_n = \int_{\Theta} \theta \pi_n(\theta|X_{1:n}) \ud\theta,
\end{align*}
the Bayes estimator of \(\theta\) in the problem. Then under Assumptions (A1)--(A4), as \(n\conv+\infty\),
\begin{align*}
n^{\frac{1}{2}} (\widetilde{\theta}_n - \widehat{\theta}_n) \convproba 0.
\end{align*}
\end{corollary}

\begin{proof}[Proof of Corollary \ref{cor.bayesestpb}]
By definition,
\begin{align*}
\widetilde{\theta}_n &= \int_{\Theta} \theta \pi_n(\theta|X_{1:n}) \ud\theta
\end{align*}
and this allows us to write
\begin{align*}
n^{\frac{1}{2}}(\widetilde{\theta}_n - \widehat{\theta}_n) &= \int_{\Theta} n^{\frac{1}{2}}(\theta - \widehat{\theta}_n) \pi_n(\theta|X_{1:n}) \ud\theta \\
&= \int_{\R^3} t \widetilde{\pi}_n(t|X_{1:n}) \ud t \convproba 0,
\end{align*}
the last convergence being a direct consequence of Theorem \ref{theo.asympposteriorpb} with \(k_0=1\).
\end{proof}

Observe that, under conditions \eqref{eq.definitiondn}, the same arguments naturally apply to the pseudo-problem and lead to a strong consistency (a.s. convergence) of its associated Bayes estimator due to Theorem \ref{theo.asympposteriorpseudo}, thus recovering the results of \cite{Ghosh} for the regularised version of the problem.

\section{Strong consistency of the MLE}\label{sec:strongconsMLE}
In this Section we prove the strong consistency of the MLE over any compact set including the true parameter (see Theorem \ref{theo.strongconsistency}). It is a prerequisite for a more accurate version of the strong consistency (see Theorem \ref{theo.strongconsistencywithrate}) which lies at the heart of the proof of Theorem \ref{theo.asympposteriorpseudo}.

\begin{theorem}\label{theo.strongconsistency}Under Assumptions (A1)--(A4), we have a.s., as \(n\conv+\infty\),
\begin{align*}
\|{\widehat\theta}_n-\theta_0\| &= \po(1).
\end{align*}
\end{theorem}

\begin{proof}[Proof of Theorem \ref{theo.strongconsistency}]
Recall that \(K\) is a compact subset of \(\Theta\), such that \(\widehat{\theta}_n \in K\) for any \(n\) large enough. We denote
\begin{align*}
l_{1:n}(X_{1:n}|S) &= \sup_{\theta\in S} l_{1:n}(X_{1:n}|\theta), \; \text{for any } S \subset K, \\
K_n(a) &= \left\{\theta\in\Theta,\; l_{1:n}(X_{1:n}|\theta) \geqslant \log a + l_{1:n}(X_{1:n}|K)\right\}, \; \text{for any } a\in\interoo{0}{1}.
\end{align*}
All we need to prove is that
\begin{align}
\exists a \in\interoo{0}{1},\; \proba\left(\lim_{n\conv+\infty} \sup_{\theta\in K_n(a)} \|\theta-\theta_0\| = 0 \right) &= 1. \label{eq.contrapositive}
\end{align}
since for any \(n\) large enough we have \({\widehat\theta}_n\in K_n(a)\) for any \(a\in\interoo{0}{1}\). We control the likelihood upon the complement of a small ball in \(K\) and prove the contrapositive of \eqref{eq.contrapositive} using compacity arguments. The extensive proof of \eqref{eq.contrapositive} is to be found in Section \ref{subsec:proof3} .
\end{proof}

We strengthen the result of Theorem \ref{theo.strongconsistency} by giving a rate of convergence for the MLE (see Theorem \ref{theo.strongconsistencywithrate}). This requires a rate of convergence for the image of the MLE through the regression function of the model, that we give in the Proposition \ref{prop.feder314} below.

\begin{proposition}\label{prop.feder314}
Under Assumptions (A1)--(A4), as \(n\conv+\infty\), a.s., for any open interval \(I\subset\interff{\underline{u}}{\overline{u}}\),
\begin{align*}
\min_{t_i\in I,\; i\leqslant n} \left|\mu(\widehat{\eta}_n, t_i) - \mu(\eta_0, t_i)\right| = \go\left(n^{-\frac{1}{2}}\log n\right).
\end{align*}
\end{proposition}

\begin{proof}[Proof of Proposition \ref{prop.feder314}]The proof is given in Section \ref{subsec:proof3}.
\end{proof}

\begin{theorem}\label{theo.strongconsistencywithrate}Under Assumptions (A1)--(A4), we have a.s., as \(n\conv+\infty\),
\begin{align}\label{eq.theo.strongconsistencywithrate}
\|{\widehat\theta}_n-\theta_0\|  &=  \go\left(n^{-\frac{1}{2}}\log n\right).
\end{align}
\end{theorem}

\begin{proof}[Proof of Theorem \ref{theo.strongconsistencywithrate}]
We show that a.s. \eqref{eq.theo.strongconsistencywithrate} holds for each coordinate of \({\widehat\theta}_n-\theta_0\). The calculations for the variance \(\sigma^2\) are pushed back into Section \ref{subsec:proof3}. We now prove the result for the parameters \(\gamma\) and \(u\). It is more convenient to use a reparametrisation of the model in terms of slope \(\gamma\) and intercept \(\beta\) where \(\beta = -\gamma u\).

\vspace{0.25em}\par\noindent\textsc{Slope \(\gamma\) and intercept \(\beta\).~}
Let \(V_1\) and \(V_2\) be two non empty open intervals of \(\interoo{\underline{u}}{u_0}\) such that their closures \(\overline{V_1}\) and \(\overline{V_2}\) do not overlap. For any \((t_1, t_2)\in V_1 \times V_2\), define \(M(t_1, t_2)\) the obviously invertible matrix
\begin{align*}
M(t_1, t_2) = \left[\begin{array}{cc}1 & t_1 \\ 1 & t_2\end{array}\right],
\end{align*}
and observe that for any \(\tau = (\beta, \gamma)\),
\begin{align*}
M(t_1, t_2) \tau = \left[\begin{array}{c}\mu(\eta, t_1) \\ \mu(\eta, t_2)\end{array}\right].
\end{align*}
Observe that by some basic linear algebra tricks we are able to write for any \((t_1, t_2)\in V_1 \times V_2\)
\begin{align*}
\|\widehat{\tau}_n - \tau_0\|_\infty &= \|M(t_1, t_2)^{-1} M(t_1, t_2)(\widehat{\tau}_n - \tau_0)\|_\infty \\
&\leqslant \||M(t_1, t_2)^{-1}\||_\infty \cdot \| M(t_1, t_2) \widehat{\tau}_n - M(t_1, t_2) \tau_0\|_\infty \\
&\leqslant \frac{|t_2|+|t_1|+2}{|t_2-t_1|} \cdot \| M(t_1, t_2) \widehat{\tau}_n - M(t_1, t_2) \tau_0 \|_\infty.
\end{align*}
Thus, using the equivalence of norms and a simple domination of the first term of the product in the inequality above, we find that there exists a constant \(C\in\R_+^*\), such that for any \((t_1, t_2)\in V_1 \times V_2\)
\begin{align}
\|\widehat{\tau}_n - \tau_0\| &\leqslant C \cdot \| M(t_1, t_2) \widehat{\tau}_n - M(t_1, t_2) \tau_0 \|, \nonumber
\intertext{i.e.}
\|\widehat{\tau}_n - \tau_0\| &\leqslant C \cdot \left[\sum_{i=1}^2 (\mu(\widehat{\eta}_n, t_i) - \mu(\eta_0, t_i))^2\right]^{\frac{1}{2}}. \label{eq.dominetaneta0}
\end{align}
Taking advantage of Proposition \ref{prop.feder314}, we are able to exhibit two sequences of points \((t_{1,n})_{n\in\N}\) in \(V_1\) and \((t_{2,n})_{n\in\N}\) in \(V_2\) such that a.s., for \(i=1,2\)
\begin{align}
\left|\mu(\widehat{\eta}_n, t_{i,n}) - \mu(\eta_0, t_{i,n})\right| = \go\left(n^{-\frac{1}{2}}\log n\right). \label{eq.consequencefeder314}
\end{align}
Combining \eqref{eq.dominetaneta0} and \eqref{eq.consequencefeder314} together (using \(t_i = t_{i,n}\) for every \(n\)), it is now trivial to see that a.s.
\begin{align*}
\|{\widehat\tau}_n-\tau_0\|  &=  \go\left(n^{-\frac{1}{2}}\log n\right),
\end{align*}
which immediately implies the result for the \(\gamma\) and \(\beta\) components of \(\theta\).

\vspace{0.25em}\par\noindent\textsc{Break-point \(u\).~}
Recalling that \(u=-\beta\gamma^{-1}\) and thanks to the result we just proved, we find that a.s.
\begin{align*}
\widehat{u}_n &= -\widehat{\beta}_n\widehat{\gamma}_n^{-1}
= -\left[\beta_0 + \go\left(n^{-\frac{1}{2}}\log n\right)\right] \left[\gamma_0 + \go\left(n^{-\frac{1}{2}}\log n\right)\right]^{-1} \\
&= -\beta_0 \gamma_0^{-1} + \go\left(n^{-\frac{1}{2}}\log n\right)
= u_0 + \go\left(n^{-\frac{1}{2}}\log n\right).
\end{align*}
\end{proof}

\section{Asymptotic distribution of the MLE}\label{sec:asympdistMLE}
In this Section we derive the asymptotic distribution of the MLE for the pseudo-problem (see Proposition \ref{prop.asympdistpseudo}) and then show that the MLE of pseudo-problem and that of the problem share the same asymptotic distribution (see Theorem \ref{theo.asympdist}).
\begin{proposition}\label{prop.asympdistpseudo}Under Assumptions (A1)--(A4) and conditions \eqref{eq.definitiondn}, as \(n\conv+\infty\)
\begin{align*}
n^{\frac{1}{2}}\left({\widehat\theta}_n^* - \theta_0\right) \convloi \loi{N}\left(0, I(\theta_0)^{-1}\right),
\end{align*}
where the asymptotic Fisher Information Matrix \(I(\cdot)\) is defined in \eqref{eq.lemmaasympFishInfMatrix}.
\end{proposition}

\begin{proof}[Proof of Theorem \ref{prop.asympdistpseudo}]
The proof is divided in two steps. We first show that the likelihood of the pseudo-problem is a.s. differentiable in a neighbourhood of the MLE \({\widehat\theta}_n^*\) for \(N\) large enough. We then recover the asymptotic distribution of the MLE following the usual scheme of proof, with a Taylor expansion of the likelihood of the pseudo-problem around the true parameter. The details of these two steps are given in Section \ref{subsec:proof4}.
\end{proof}

\begin{theorem}\label{theo.asympdist}
Under Assumptions (A1)--(A4) and conditions \eqref{eq.definitiondn}, as \(n\conv+\infty\),
\begin{align*}
n^{\frac{1}{2}}\left({\widehat\theta}_n - \theta_0\right) \convloi \loi{N}\left(0, I(\theta_0)^{-1}\right),
\end{align*}
where the asymptotic Fisher Information Matrix \(I(\cdot)\) is defined in \eqref{eq.lemmaasympFishInfMatrix}.
\end{theorem}

\begin{proof}[Proof of Theorem \ref{theo.asympdist}]
It is a direct consequence of Proposition \ref{prop.asympdistpseudo} as soon as we show that as \(n\conv+\infty\)
\begin{align}
{\widehat\theta}_n - {\widehat\theta}_n^* = \pop\left(n^{-\frac{1}{2}}\right). \label{eq.proximityproba}
\end{align}
To prove \eqref{eq.proximityproba}, we study each coordinate separately. For \(\gamma\) and \(u\), we apply Lemmas 4.12 and 4.16 found in \cite{Feder} with a slight modification: the rate of convergence \(d_n\) he uses may differ from ours but it suffices to formally replace \((\log\log n)^{\frac{1}{2}}\) by \((\log n)\) all throughout his paper and the proofs he provides go through without any other change. We thus get
\begin{align}
{\widehat\gamma}_n - {\widehat\gamma}_n^* &= \pop\left(n^{-\frac{1}{2}}\right), 
&{\widehat u}_n - {\widehat u}_n^* &= \pop\left(n^{-\frac{1}{2}}\right).  \label{eq.asymppseudogammau}
\end{align}
It now remains to show that
\begin{align}
\widehat{\sigma}_n^2 - \widehat{\sigma}_n^{2*} &= \pop\left(n^{-\frac{1}{2}}\right). \label{eq.asymppseudosigma2}
\end{align}
To do so, we use \eqref{eq.asymppseudogammau} and the decomposition \eqref{eq.sigma2ai}
\begin{align*}
\widehat{\sigma}_n^2 &= \frac{1}{n} \sum_{i=1}^n \nu_i^2(\widehat{\eta}_n) + \frac{2}{n} \sum_{i=1}^n \nu_i(\widehat{\eta}_n) \xi_i + \frac{1}{n} \sum_{i=1}^n \xi_i^2,
\end{align*}
where \(\nu_i(\widehat{\eta}_n) = \gamma_0 \cdot (t_i - u_0) \one_{\interfo{t_i}{+\infty}}(u_0) - \widehat{\gamma}_n \cdot (t_i - \widehat{u}_n) \one_{\interfo{t_i}{+\infty}}(\widehat{u}_n)\). The details of this are available in Section \ref{subsec:proof4}.
\end{proof}

\section{Discussion}\label{sec:discussion}
In this Section, we summarise the results presented in this paper. The consistency of the posterior distribution for a piecewise linear regression model is derived as well as its asymptotic normality with suitable normalisation. The proofs of these convergence results rely on the convergence of the MLE which is also proved here. In order to obtain all the asymptotic results, a regularised version of the problem at hand, called pseudo-problem, is first studied and the difference between this pseudo-problem and the (full) problem is then shown to be asymptotically negligeable.

The trick of deleting observations in a diminishing neighbourhood of the true parameter, originally found in \cite{Sylwester} allows the likelihood of the pseudo-problem to be differentiated at the MLE, once the MLE is shown to asymptotically belong to that neighbourhood (this requires at least a small control of the rate of convergence of the MLE). This is the key argument needed to derive the asymptotic distribution of the MLE through the usual Taylor expansion of the likelihood at the MLE. Extending the results of \cite{Ghosh} to a non i.i.d. setup, the asymptotic normality of the posterior distribution for the pseudo-problem is then recovered from that of the MLE, and passes on almost naturally to the (full) problem.

The asymptotic normality of the MLE and the posterior distribution are proved in this paper in a non i.i.d. setup with a non continuously differentiable likelihood. In both cases we obtain the same asymptotic results as for an i.i.d. regular model: the rate of convergence is \(\sqrt{n}\) and the limiting distribution is Gaussian \cite[see][]{Ghosh,Lehmann}. For the piecewise linear regression model, the exogenous variable \(t_{1:n}\) does not appear in the expression of the rate of convergence as opposed to what is known for the usual linear regression model \cite[see][]{Lehmann}: this is due to our own Assumption (A1) which implies that \(t_{1:n}^\prime t_{1:n}\) is equivalent to \(n\). Note that for a simple linear regression model, we also obtain the rate \(\sqrt{n}\) under Assumption (A1). In the litterature, several papers already highlighted the fact that the rate of convergence and the limiting distribution (when it exists) may be different for non regular models in the sense that the likelihood is either non continuous, or non continuously differentiable, or admits singularities \cite[see][]{Dacunha,GhoshGhosalSamanta,GhosalSamanta,Ibragimov}. For the piecewise regression model, the likelihood is continuous but non continuously differentiable on a countable set (but the left and right derivatives exist and are finite): the rate of convergence \(\sqrt{n}\) is not so surprising in our case, because this rate was already obtained for a univariate i.i.d. model the likelihood of has the same non regularity at a single point. In that case, the rate of convergence of the MLE is shown to be \(n\) \cite[see][for instance]{Dacunha}.

\section{Extensive proofs}\label{sec:proofmainresults}
\subsection{Proofs of Section \ref{sec:bayescons}}\label{subsec:proof2}
\begin{proof}[\underline{Proof of Theorem \ref{theo.strongconsistencyposterior}}]
To prove \eqref{eq.fractiontrongconsistencyposterior}, we proceed as announced and deal with numerator and denominator in turn.
\vspace{0.25em}\par\noindent\textsc{Majoration.~} From Proposition \ref{prop.majorationdiffdeL} with \(\rho_n = 1\), for any given \(\epsilon > 0\), we can choose \(\delta>0\) small enough so that a.s. for any \(n\) large enough
\begin{align*}
\sup_{\theta \in B^c(\theta_0, \delta)} \frac{1}{n}[l_{1:n}(X_{1:n}|\theta)-l_{1:n}(X_{1:n}|\theta_0)] &\leqslant -\epsilon.
\end{align*}
We thus obtain a.s. for any \(n\) large enough
\begin{align}
0&\leqslant \int_{B^c(\theta_0, \delta)} \pi(\theta) \exp[l_{1:n}(X_{1:n}|\theta)-l_{1:n}(X_{1:n}|\theta_0)] \ud \theta \nonumber \\
&\leqslant e^{-n\epsilon} \int_{B^c(\theta_0, \delta)} \pi(\theta) \ud \theta. \label{eq.majorationnumstrongconsistencyposterior}
\end{align}

\vspace{0.25em}\par\noindent\textsc{Minoration.~} Define \(\theta_n \in \overline{B(\theta_0, \delta)} \) such that
\begin{align*}
\inf_{\theta \in B(\theta_0, \delta)} \frac{1}{n}[l_{1:n}(X_{1:n}|\theta)-l_{1:n}(X_{1:n}|\theta_0)] &= \frac{1}{n}[l_{1:n}(X_{1:n}|\theta_n)-l_{1:n}(X_{1:n}|\theta_0)]
\end{align*}
It is possible to define such a \(\theta_n\) because \(\overline{B(\theta_0, \delta)}\) is a compact subset of \(\Theta\) for \(\delta>0\) small enough and \(l_{1:n}(X_{1:n}|\cdot)\) is continuous as a function of \(\theta\).
Let now
\begin{align}
b_n(\theta) &= \left(\frac{\sigma_0^2}{\sigma^2}-1-\log\frac{\sigma_0^2}{\sigma^2}\right) + \frac{1}{\sigma^2}\cdot\frac{1}{n}\sum_{i=1}^{n}\left[\mu(\eta_0, t_i) - \mu(\eta, t_i)\right]^2.\label{eq.definitionsylwesterbn}
\end{align}
Recalling the definition of the log-likehood given in \eqref{eq.defloglikelihood1n} and replacing \(X_i\) by its expression given in \eqref{eq.modelesylwester} we find via straightforward algebra
\begin{align}
&~\dfrac{2}{n}[l_{1:n}(X_{1:n}|\theta) - l_{1:n}(X_{1:n}|\theta_0)] = \log\frac{\sigma_0^2}{\sigma^2} + \left( \frac{1}{\sigma^2}-\frac{1}{\sigma_0^2}\right)\left(\frac{1}{n}\sum_{i=1}^n \xi_i^2\right) \nonumber \\
&\qquad -\frac{1}{n\sigma^2} \sum_{i=1}^{n}[\mu(\eta_0, t_i) - \mu(\eta, t_i)]^2  -\frac{2}{\sigma^2}\frac{1}{n}\sum_{i=1}^n [\mu(\eta_0, t_i) - \mu(\eta, t_i)] \xi_i \nonumber\\
&= \log\frac{\sigma_0^2}{\sigma^2} + \left( \frac{1}{\sigma^2}-\frac{1}{\sigma_0^2}\right)\left(\frac{1}{n}\sum_{i=1}^n \xi_i^2 - \sigma_0^2 + \sigma_0^2\right)\nonumber \\
&\qquad -\frac{1}{n\sigma^2} \sum_{i=1}^{n}[\mu(\eta_0, t_i) - \mu(\eta, t_i)]^2  -\frac{2}{\sigma^2}\frac{1}{n}\sum_{i=1}^n [\mu(\eta_0, t_i) - \mu(\eta, t_i)] \xi_i \nonumber \\
&= \left(\log\frac{\sigma_0^2}{\sigma^2} + 1 - \frac{\sigma_0^2}{\sigma^2}\right) + \frac{\sigma_0^2-\sigma^2}{\sigma^2\sigma_0^2}\left(\frac{1}{n}\sum_{i=1}^n \xi_i^2 - \sigma_0^2\right)\nonumber \\
&\qquad  -\frac{1}{n\sigma^2} \sum_{i=1}^{n}[\mu(\eta_0, t_i) - \mu(\eta, t_i)]^2 -\frac{2}{\sigma^2}\frac{1}{n}\sum_{i=1}^n [\mu(\eta_0, t_i) - \mu(\eta, t_i)] \xi_i \label{eq.diffdelog1} \\
&= -b_n(\theta) + \frac{\sigma_0^2-\sigma^2}{\sigma^2\sigma_0^2}\left(\frac{1}{n}\sum_{i=1}^n \xi_i^2 - \sigma_0^2\right) - \frac{2}{\sigma^2}\frac{1}{n}\sum_{i=1}^n [\mu(\eta_0, t_i) - \mu(\eta, t_i)] \xi_i. \label{eq.diffdelog2}
\end{align}
It is now easy to see that
\begin{align*}
&~\inf_{\theta \in B(\theta_0, \delta)} \frac{2}{n}[l_{1:n}(X_{1:n}|\theta)-l_{1:n}(X_{1:n}|\theta_0)] = \frac{2}{n}[l_{1:n}(X_{1:n}|\theta_n)-l_{1:n}(X_{1:n}|\theta_0)] \\
&= -b_n(\theta_n) + \frac{\sigma_0^2-\sigma_n^2}{\sigma_n^2\sigma_0^2}\left(\frac{1}{n}\sum_{i=1}^n \xi_i^2 - \sigma_0^2\right) - \frac{2}{\sigma_n^2}\frac{1}{n}\sum_{i=1}^n [\mu(\eta_0, t_i) - \mu(\eta_n, t_i)] \xi_i \\
&= -b_n(\theta_n) + \frac{1}{\sigma_n^2} \left[\frac{\sigma_0^2-\sigma_n^2}{\sigma_0^2} \left(\frac{1}{n}\sum_{i=1}^n \xi_i^2 - \sigma_0^2\right) - \frac{2}{n}\sum_{i=1}^n [\mu(\eta_0, t_i) - \mu(\eta_n, t_i)] \xi_i\right] \\
&= -b_n(\theta_n) + \frac{1}{\sigma_n^2} R_n \\
&= \left(\log\frac{\sigma_0^2}{\sigma_n^2} + 1 - \frac{\sigma_0^2}{\sigma_n^2}\right) - \frac{1}{\sigma_n^2}\cdot\frac{1}{n}\sum_{i=1}^{n}\left[\mu(\eta_0, t_i) - \mu(\eta_n, t_i)\right]^2  + \frac{1}{\sigma_n^2} R_n
\end{align*}
where \(R_n\convps 0\) because of the Law of Large Numbers and Lemma \ref{lemma.sylwestertheo35}. Thanks to Lemma \ref{lemma.convergenceuniformemu} we thus find that there exists \(C\in\R_+^*\) such that
\begin{align*}
\inf_{\theta \in B(\theta_0, \delta)} \frac{2}{n}[l_{1:n}(X_{1:n}|\theta)-l_{1:n}(X_{1:n}|\theta_0)]
&\geqslant \left(\log\frac{\sigma_0^2}{\sigma_n^2} + 1 - \frac{\sigma_0^2}{\sigma_n^2}\right) \\
&\qquad - \frac{1}{\sigma_n^2}\left(C\|\theta_n-\theta_0\|^2-R_n\right)
\end{align*}
We now choose \(\kappa>0\) and \(\delta>0\) small enough so that
\begin{align}
\sigma_n^2 \left(\log\frac{\sigma_0^2}{\sigma_n^2} + 1 - \frac{\sigma_0^2}{\sigma_n^2}\right) &\geqslant -\kappa, \label{eq.kappadelta1} \\
-\frac{3(\kappa + C \delta^2)}{2(\sigma_0^2-\delta)} \geqslant -\frac{1}{2}\epsilon. \label{eq.kappadelta2}
\end{align}
Thanks to \eqref{eq.kappadelta1} and the definition of \(\theta_n\), we can now write that
\begin{align*}
\inf_{\theta \in B(\theta_0, \delta)} \frac{2}{n}[l_{1:n}(X_{1:n}|\theta)-l_{1:n}(X_{1:n}|\theta_0)]
&\geqslant - \frac{1}{\sigma_n^2}\left(\kappa + C\|\theta_n-\theta_0\|^2-R_n\right)\\
&\geqslant - \frac{1}{\sigma_n^2}\left(\kappa + C \delta^2 - R_n\right).
\end{align*}
Since for any \(n\) large enough
\begin{align*}
|R_n| \leqslant \frac{1}{2}\left(\kappa + C \delta^2\right),
\end{align*}
we find via \eqref{eq.kappadelta2} that for any \(n\) large enough
\begin{align*}
\inf_{\theta \in B(\theta_0, \delta)} \frac{2}{n}[l_{1:n}(X_{1:n}|\theta)-l_{1:n}(X_{1:n}|\theta_0)]
&\geqslant - \frac{3}{2\sigma_n^2}\left(\kappa + C \delta^2\right) \\
&\geqslant - \frac{3(\kappa + C \delta^2)}{2(\sigma_0^2-\delta)}
\geqslant -\frac{1}{2}\epsilon.
\end{align*}
We just proved that for any \(\epsilon>0\), we have a.s. for any \(n\) large enough
\begin{align*}
0 &\geqslant \inf_{\theta \in B(\theta_0, \delta)} \frac{2}{n}[l_{1:n}(X_{1:n}|\theta)-l_{1:n}(X_{1:n}|\theta_0)] \geqslant -\frac{1}{2}\epsilon,
\end{align*}
which immediately implies
\begin{align}
\int_{B(\theta_0, \delta)} \pi(\theta) \exp[l_{1:n}(X_{1:n}|\theta)-l_{1:n}(X_{1:n}|\theta_0)] \ud \theta &\geqslant e^{-\frac{1}{2}n\epsilon} \int_{B(\theta_0, \delta)} \pi(\theta) \ud \theta. \label{eq.minorationnumstrongconsistencyposterior}
\end{align}
\vspace{0.25em}\par\noindent\textsc{Conclusion.~}
Let now \(\epsilon>0\) and \(\delta>0\) small enough so that a.s. for any \(n\) large enough \eqref{eq.majorationnumstrongconsistencyposterior} and  \eqref{eq.minorationnumstrongconsistencyposterior} both hold. We have a.s. for any \(n\) large enough
\begin{align*}
\dfrac{\int_{B^c(\theta_0, \delta)} \pi(\theta) \exp[l_{1:n}(X_{1:n}|\theta)-l_{1:n}(X_{1:n}|\theta_0)] \ud \theta}{\int_{B(\theta_0, \delta)} \pi(\theta) \exp[l_{1:n}(X_{1:n}|\theta)-l_{1:n}(X_{1:n}|\theta_0)] \ud \theta} &\leqslant  \frac{\int_{B^c(\theta_0, \delta)} \pi(\theta) \ud \theta}{\int_{B(\theta_0, \delta)} \pi(\theta) \ud \theta} e^{-\frac{1}{2}n\epsilon} \conv 0,
\end{align*}
which ends the proof.
\end{proof}

\begin{proof}[\underline{Proof of Theorem \ref{theo.asympposteriorpseudo}}]
Because the posterior distribution of \(\theta\) in the pseudo-problem, \(\pi_n^*(\cdot|X_{1:n})\), can be written as
\begin{align*}
\pi_n^*(\theta|X_{1:n}) \propto \pi(\theta) \exp[l_{1:n}^*(X_{1:n}|\theta)],
\end{align*}
the posterior density of \(t^* = n^{\frac{1}{2}} (\theta-\widehat{\theta}_n^*) \in \R^3\) can be written as
\begin{align*}
\widetilde{\pi}_n^*(t|X_{1:n}) &= C_n^{-1} \pi(\widehat{\theta}_n^*+n^{-\frac{1}{2}}t)\exp[l_{1:n}^*(X_{1:n}|\widehat{\theta}_n^*+n^{-\frac{1}{2}}t)-l_{1:n}^*(X_{1:n}|\widehat{\theta}_n^*)]
\end{align*}
where
\begin{align}\label{eq.definitionconstanteCn}
C_n &= \int_{\R^3} \pi(\widehat{\theta}_n^*+n^{-\frac{1}{2}}t)\exp[l_{1:n}^*(X_{1:n}|\widehat{\theta}_n^*+n^{-\frac{1}{2}}t)-l_{1:n}^*(X_{1:n}|\widehat{\theta}_n^*)] \ud t.
\end{align}
Denoting
\begin{align}\label{eq.definitionfonctiongn}
g_n(t) &= \pi(\widehat{\theta}_n^*+n^{-\frac{1}{2}}t)\exp[l_{1:n}^*(X_{1:n}|\widehat{\theta}_n^*+n^{*-\frac{1}{2}}t)-l_{1:n}^*(X_{1:n}|\widehat{\theta}_n^*)] \nonumber\\
&\qquad- \pi(\theta_0) e^{-\frac{1}{2}t^\prime I(\theta_0) t},
\end{align}
to prove \eqref{eq.asympposteriorpseudo} it suffices to show that for any \(0\leqslant k\leqslant k_0\),
\begin{align}\label{eq.integralegnconvvers0}
\int_{\R^3} \|t\|^k |g_n(t)|\ud t \convps 0.
\end{align}
Indeed, if \eqref{eq.integralegnconvvers0} holds, \(C_n \convps \pi(\theta_0)(2\pi)^{\frac{3}{2}}|I(\theta_0)|^{-\frac{1}{2}}\) (\(k=0\)) and therefore, the integral in \eqref{eq.asympposteriorpseudo} which is dominated by
\begin{align*}
&C_n^{-1} \int_{\R^3} \|t\|^k |g_n(t)|\ud t \\
&\qquad + \int_{\R^3} \|t\|^k \left|C_n^{-1}\pi(\theta_0) e^{-\frac{1}{2}t^\prime I(\theta_0) t} - (2\pi)^{-\frac{1}{2}}|I(\theta_0)|^{\frac{1}{2}}e^{-\frac{1}{2}t^\prime I(\theta_0) t}\right|\ud t
\end{align*}
also goes to zero a.s.

Let \(0<\delta\) to be chosen later, and let \(0\leqslant k\leqslant k_0\). To show \eqref{eq.integralegnconvvers0}, we break \(\R^3\) into two regions
\begin{align*}
T_1(\delta) &= B^c(0, \delta n^{\frac{1}{2}}d_n) = \{t: \|t\| \geqslant \delta n^{\frac{1}{2}}d_n \} \\
T_2(\delta) &= B(0, \delta n^{\frac{1}{2}}d_n) = \{t: \|t\| < \delta n^{\frac{1}{2}}d_n \}
\end{align*}
and show that for \(i=1,2\)
\begin{align}\label{eq.integralegnconvpsvers0surTi}
\int_{T_i(\delta)} \|t\|^k |g_n(t)| \ud t \convps 0.
\end{align}

\vspace{0.25em}\par\noindent\textsc{Proof for \(i=1\).~} Note that \(\int_{T_1(\delta)} \|t\|^k |g_n(t)|\) is dominated by
\begin{align*}
&\int_{T_1(\delta)} \|t\|^k \pi(\widehat{\theta}_n^*+n^{\frac{1}{2}}t) \exp[l_{1:n}^*(X_{1:n}|\widehat{\theta}_n^*+n^{-\frac{1}{2}}t)-l_{1:n}^*(X_{1:n}|\widehat{\theta}_n^*)] \ud t\\
&\qquad + \int_{T_1(\delta)} \|t\|^k \pi(\theta_0) e^{-\frac{1}{2}t^\prime I(\theta_0) t} \ud t.
\end{align*}
The second integral trivially goes to zero. For the first integral, we observe that it can be rewritten as
\begin{align*}
n^{\frac{1}{2}} \int_{B^c(\widehat{\theta}_n^*, \delta d_n)} n^\frac{k}{2}\|\theta-\widehat{\theta}_n^*\|^k \pi(\theta) \exp[l_{1:n}^*(X_{1:n}|\theta)-l_{1:n}^*(X_{1:n}|\widehat{\theta}_n^*)] \ud \theta.
\end{align*}
The strong consistency of \(\widehat{\theta}_n^*\) (see Theorem \ref{theo.strongconsistencywithrate}) implies that a.s., for any \(n\) large enough
\begin{align*}
\|\widehat{\theta}_n^*-\theta_0\|<\frac{1}{2}\delta d_n.
\end{align*}
From this, we deduce that a.s., for any \(n\) large enough, \(B^c(\widehat{\theta}_n^*, \delta d_n) \subset B^c(\theta_0, \frac{1}{2}\delta d_n)\) and thus that the first integral is dominated by
\begin{align*}
n^{\frac{k+1}{2}}\int_{B^c(\theta_0, \frac{1}{2}\delta d_n)} \|\theta-\widehat{\theta}_n^*\|^k \pi(\theta) \exp[l_{1:n}^*(X_{1:n}|\theta)-l_{1:n}^*(X_{1:n}|\widehat{\theta}_n^*)] \ud \theta.
\end{align*}
Recalling that \(n^*\sim n\), Proposition \ref{prop.majorationdiffdeL} with \(\rho_n = d_n\) implies that there a.s. exists \(\epsilon > 0\) such that for any \(n\) large enough and any \(\theta\in B^c(\theta_0, \frac{1}{2}\delta d_n)\) we have
\begin{align*}
l_{1:n}^*(X_{1:n}|\theta) - l_{1:n}^*(X_{1:n}|\widehat{\theta}_n^*) \leqslant -\epsilon n d_n^2.
\end{align*}
It follows, using \eqref{eq.definitiondn} that, a.s. for any \(n\) large enough the first integral is dominated by
\begin{align*}
n^{\frac{k+1}{2}} \exp(-\epsilon n d_n^2) \int_\Theta \|\theta-\widehat{\theta}_n^*\|^k \pi(\theta) \ud t
&=n^{\frac{k+1}{2}} \exp(-\epsilon n d_n^2) \cdot \go(1)  \\
&\leqslant n^{\frac{k+1}{2}} n^{-\epsilon \log n}\cdot \go(1) \conv 0,
\end{align*}
since by \eqref{eq.definitiondn} we find that \(n d_n^2 \geqslant (\log n)^2\) for any \(n\) large enough. Hence \eqref{eq.integralegnconvpsvers0surTi} holds for \(i=1\).

\vspace{0.25em}\par\noindent\textsc{Proof for \(i=2\).~}
We first recall the multivariate Taylor expansion for a function \(g\) (k+1)-times continuously differentiable within a neighbourhood of \(y\in\R^n\). With the usual differential calculus notations
\begin{align*}
\uD^\alpha g(y) \cdot h^{(\alpha)} = \sum_{1\leqslant i_1,\ldots,i_\alpha\leqslant n} \frac{\partial^\alpha g}{\partial_{i_1}\cdots\partial_{i_\alpha}} (y) \cdot h_{i_1}\cdots h_{i_\alpha}
\end{align*}
we have
\begin{align}\label{eq.taylormultivarie}
g(x) = \sum_{\alpha=0}^k \frac{1}{\alpha!}\uD^\alpha g(y)\cdot(x-y)^{(\alpha)} + R_{k+1}(x)
\end{align}
where
\begin{align}\label{eq.taylormultivarieremainder}
R_{k+1}(x) = \frac{1}{(k+1)!} \int_{0}^{1} (1-s)^{k} \uD^{k+1} g(y+s(x-y)) \cdot (x-y)^{(k+1)} \ud s.
\end{align}

Before expanding the log-likelihood over \(T_2(\delta)\) in a such a way, we first have to make sure it is differentiable over the correct domain. Indeed, the strong consistency of \(\widehat{\theta}_n^*\) (see Theorem \ref{theo.strongconsistencywithrate}) implies that a.s., whatever \(\delta_0 >0\), for \(n\) large enough,
\begin{align*}
\|\widehat{\theta}_n^* - \theta_0\| < \delta_0 d_n.
\end{align*}
For \(\delta\) chosen small enough, since \(t\in T_2(\delta)\) implies
\begin{align*}
\|\theta - \widehat{\theta}_n^*\| < \delta d_n
\end{align*}
it follows from the triangle inequality that a.s. for \(n\) large enough,
\begin{align*}
\|\theta - \theta_0\| < (\delta_0 + \delta) d_n < d_n.
\end{align*}
A.s. for any \(n\) large enough, \(t\in T_2(\delta)\) hence implies \(\theta\in B(\theta_0, (\delta+\delta_0) d_n)\). We choose \(\delta_0\) and \(\delta\) small enough so that \(\delta +\delta_0 < 1\). This way, \(\theta\mapsto l_{1:n}^*(X_{1:n}|\theta)\) is guaranteed to be infinitely continuously differentiable over \(B(\theta_0, (\delta+\delta_0) d_n) \subset B(\theta_0, d_n)\).

Now expanding the log-likelihood in a Taylor series for any \(n\) large enough, and taking advantage of the fact that \(l_{1:n}^*(X_{1:n}|\widehat{\theta}_n^*) = 0\), we define \(B_{1:n}^*(\cdot)\) the symmetric matrix defined for \(u\in D_n\) by
\begin{align}
B_{1:n}^*(\theta) &= -
\left[\begin{array}{ccc}
\displaystyle \frac{\partial^2 l_{1:n}^*(X_{1:n}|\theta)}{\partial \gamma \partial \gamma} & \displaystyle \frac{\partial^2 l_{1:n}^*(X_{1:n}|\theta)}{\partial \gamma \partial u} & \displaystyle \frac{\partial^2 l_{1:n}^*(X_{1:n}|\theta)}{\partial \gamma \partial \sigma^2} \\
& \displaystyle \frac{\partial^2 l_{1:n}^*(X_{1:n}|\theta)}{\partial u \partial u} & \displaystyle \frac{\partial^2 l_{1:n}^*(X_{1:n}|\theta)}{\partial u \partial \sigma^2} \\
& & \displaystyle \frac{\partial^2 l_{1:n}^*(X_{1:n}|\theta)}{\partial \sigma^2 \partial \sigma^2}
\end{array}\right].\label{eq.definitionBnstar}
\end{align}
and write that
\begin{align}\label{eq.taylorloglikelihood}
l_{1:n}^*(X_{1:n}|\theta)-l_{1:n}^*(X_{1:n}|\widehat{\theta}_n^*) &= -\frac{1}{2}(\theta-\widehat{\theta}_n^*)^\prime\left(B_{1:n}^*(\widehat{\theta}_n^*)\right) (\theta-\widehat{\theta}_n^*) \nonumber\\
&\qquad +  R_{3,n}(\theta)
\end{align}
where
\begin{align}\label{eq.taylorloglikelihoodremainder}
R_{3,n}(\theta) = \frac{1}{3!} \int_{0}^{1} (1-s)^{2} \uD^3 l_{1:n}^*(X_{1:n}|\widehat{\theta}_n^*+s(\theta-\widehat{\theta}_n^*)) \cdot (\theta-\widehat{\theta}_n^*)^{(3)} \ud s.
\end{align}
Lemma \ref{lemma.dominationloglikethirdderiv} allows us to write that a.s. there exists a constant \(C\in R_+^*\) such that for any \(n\) large enough, for any \(t\in T_2(\delta)\)
\begin{align}\label{eq.taylorloglikelihood2}
l_{1:n}^*(X_{1:n}|\widehat{\theta}_n^*+n^{-\frac{1}{2}}t)-l_{1:n}^*(X_{1:n}|\widehat{\theta}_n^*) = -\frac{1}{2}t^\prime\left(n^{-1} B_{1:n}^*(\widehat{\theta}_n^*)\right) t + S_n(t)
\end{align}
where
\begin{align}\label{eq.taylorloglikelihoodremainder2}
|S_n(t)| \leqslant C n^{-\frac{1}{2}} \cdot \|t\|^3.
\end{align}

From \eqref{eq.taylorloglikelihoodremainder2}, we obtain that for any \(t\in T_2(\delta)\), \(S_n(t)\convps 0\). Because of Lemma \ref{lemma.Bn}, we have \(n^{-1} B_{1:n}^*(\widehat{\theta}_n^*) \convps I(\theta_0)\), and it follows immediately that for any \(t\in T_2(\delta)\),
\begin{align*}
g_n(t)\convps 0,
\end{align*}
and thus that
\begin{align*}
\|t\|^k g_n(t)\convps 0.
\end{align*}
From \eqref{eq.taylorloglikelihoodremainder2} we also obtain
\begin{align*}
|S_n(t)| &\leqslant C \delta d_n \|t\|^2.
\end{align*}
Lemma \ref{lemma.Bn}, combined with \eqref{eq.definitiondn}, \eqref{eq.taylorloglikelihood2} and the positivity of \(I(\theta_0)\), ensures that a.s. for any \(n\) large enough
\begin{align*}
|S_n(t)| \leqslant \frac{1}{4}t^\prime \left(n^{-1} B_{1:n}^*(\widehat{\theta}_n^*)\right) t,
\end{align*}
so that from \eqref{eq.taylorloglikelihood2}, a.s. for any \(n\) large enough
\begin{align}
\exp[l_{1:n}^*(X_{1:n}|\widehat{\theta}_n^*+n^{-\frac{1}{2}}t)-l_{1:n}^*(X_{1:n}|\widehat{\theta}_n^*)] \leqslant e^{-\frac{1}{4}t^\prime\left(n^{-1} B_{1:n}^*(\widehat{\theta}_n^*)\right) t} \leqslant e^{-\frac{1}{8}t^\prime I(\theta_0) t}.
\end{align}
Therefore, for \(n\) large enough, \(\|t\|^k |g_n(t)|\) is dominated by an integrable function on the set \(T_2(\delta)\) and \eqref{eq.integralegnconvpsvers0surTi} holds for \(i=2\) which completes the proof.
\end{proof}

\subsection{Proofs of Section \ref{sec:strongconsMLE}}\label{subsec:proof3}
\begin{proof}[\underline{Proof of Theorem \ref{theo.strongconsistency}}]
From \eqref{eq.diffdelog2}, it is easy to see that
\begin{align*}
\dfrac{2}{n}[l_{1:n}(X_{1:n}|\theta) - l_{1:n}(X_{1:n}|K)] &\leqslant \dfrac{2}{n}[l_{1:n}(X_{1:n}|\theta) - l_{1:n}(X_{1:n}|\theta_0)] \\
&\leqslant-b_n(\theta) + \frac{\sigma_0^2-\sigma^2}{\sigma^2\sigma_0^2}\left(\frac{1}{n}\sum_{i=1}^n \xi_i^2 - \sigma_0^2\right) \\
&\qquad  - \frac{2}{\sigma^2}\frac{1}{n}\sum_{i=1}^n [\mu(\eta_0, t_i) - \mu(\eta, t_i)] \xi_i.
\end{align*}
For any \(\theta^\prime\in \Theta\) and \(r>0\), let \(B(\theta^\prime, r) = \{\theta,;\; \|\theta^\prime-\theta\|_1 < r\}\). It is now obvious that
\begin{align}
&~\dfrac{2}{n}[l_{1:n}(X_{1:n}|B(\theta^\prime, r))-l_{1:n}(X_{1:n}|K)] \nonumber\\
&\leqslant \sup_{\theta\in B(\theta^\prime, r)} \left\{-b_n(\theta)\right\} + \sup_{\theta\in B(\theta^\prime, r)} \left| \frac{\sigma_0^2-\sigma^2}{\sigma^2\sigma_0^2} \right| \cdot \left|\frac{1}{n}\sum_{i=1}^n \xi_i^2 - \sigma_0^2\right| \nonumber \\
&\qquad + \sup_{\theta\in B(\theta^\prime, r)} \left\{\frac{2}{\sigma^2}\right\} \cdot \sup_{\theta\in B(\theta^\prime, r)} \left\{\left|\frac{1}{n}\sum_{i=1}^n [\mu(\eta_0, t_i) - \mu(\eta, t_i)] \xi_i\right|\right\}. \label{eq.sylwesterps}
\end{align}
Lemma \ref{lemma.sylwestertheo35} now ensures that
\begin{align*}
\sup_{\theta\in B(\theta^\prime, r)} \left|\frac{1}{n}\sum_{i=1}^n [\mu(\eta_0, t_i) - \mu(\eta, t_i)] \xi_i\right| &\convps 0,
\end{align*}
and \(\sigma^2\) being bounded away from 0 ensures the boundedness of \(\displaystyle \sup_{\theta\in B(\theta^\prime, r)} \left\{\frac{2}{\sigma^2}\right\}\) which implies
\begin{align*}
\sup_{\theta\in B(\theta^\prime, r)} \left\{\frac{2}{\sigma^2}\right\} \cdot \sup_{\theta\in B(\theta^\prime, r)} \left\{\left|\frac{1}{n}\sum_{i=1}^n [\mu(\eta_0, t_i) - \mu(\eta, t_i)] \xi_i\right|\right\} &\convps 0.
\end{align*}
Since \(\sigma^2\) is bounded away from 0, taking advantage of the Strong Law of Large Numbers, we also obtain
\begin{align*}
\sup_{\theta\in B(\theta^\prime, r)} \left| \frac{\sigma^2-\sigma_0^2}{\sigma^2\sigma_0^2} \right| \cdot \left|\frac{1}{n}\sum_{i=1}^n \xi_i^2 - \sigma_0^2\right| &\convps 0.
\end{align*}
We may thus rewrite \eqref{eq.sylwesterps} as
\begin{align}
\dfrac{2}{n}[l_{1:n}(X_{1:n}|B(\theta^\prime, r))-l_{1:n}(X_{1:n}|K)] &\leqslant \sup_{\theta\in B(\theta^\prime, r)} \left\{-b_n(\theta) \right\}+ R_n, \label{eq.sylwesterps3}
\end{align}
where \(R_n\convps 0\).

Assume now that \(\theta^\prime\neq\theta_0\), then we have
\begin{align}
\sup_{\theta\in B(\theta^\prime, r)}|b_n(\theta)-b(\theta^\prime)| &\leqslant \sup_{\theta\in B(\theta^\prime, r)} |b_n(\theta)-b_n(\theta^\prime)| + |b_n(\theta^\prime)-b(\theta^\prime)|. \label{ineq.majbn1}
\end{align}
Lemma \ref{lemma.sylwesterbnandb} (see \eqref{eq.sylwesterbnandb3}) ensures the existence of a \(r\) small enough, say \(r=r(\theta^\prime)\), such that
\begin{align}
\sup_{\theta\in B(\theta^\prime, r(\theta^\prime))} |b_n(\theta)-b_n(\theta^\prime)| &\leqslant \frac{1}{4} b(\theta^\prime), \label{ineq.majbn2}
\end{align}
uniformly in \(n\). For \(n\) large enough, that same Lemma \ref{lemma.sylwesterbnandb} (see \eqref{eq.sylwesterbnandb4}) also guarantees that
\begin{align}
|b_n(\theta^\prime)-b(\theta^\prime)| &\leqslant \frac{1}{4} b(\theta^\prime). \label{ineq.majbn3}
\end{align}
Adding inequalities \eqref{ineq.majbn2} and \eqref{ineq.majbn3} together and combining the result with \eqref{ineq.majbn1}, we deduce that for any \(n\) large enough
\begin{align*}
\sup_{\theta\in B(\theta^\prime, r(\theta^\prime))}|b_n(\theta)-b(\theta^\prime)| &\leqslant \frac{1}{2} b(\theta^\prime),
\intertext{i.e.}
\sup_{\theta\in B(\theta^\prime, r(\theta^\prime))}\left\{-b_n(\theta)\right\} &\leqslant -\frac{1}{2} b(\theta^\prime),
\end{align*}
which finally gives together with \eqref{eq.sylwesterps3}
\begin{align}
\forall\theta^\prime \neq \theta_0, \; \proba\left(\limsup_{n\conv+\infty}\dfrac{1}{n}[l_{1:n}(X_{1:n}|B(\theta^\prime, r(\theta^\prime)))-l_{1:n}(X_{1:n}|K)] \leqslant -\frac{1}{4}b(\theta^\prime)\right) &= 1. \label{eq.limsupdifflikeetb}
\end{align}
Since Lemma \ref{lemma.sylwesterbnandb} ensures that \(b(\theta^\prime) > 0\) for any \(\theta^\prime \neq \theta_0\), the previous statement implies
\begin{align}
\forall\theta^\prime \neq \theta_0, \; \proba\left(\exists n(\theta^\prime)\in\N,\;\forall n>n(\theta^\prime), \; l_{1:n}(X_{1:n}|B(\theta^\prime, r(\theta^\prime)))-l_{1:n}(X_{1:n}|K) < -1 \right) &= 1. \label{eq.sylwesterps4}
\end{align}
For a given \(\delta>0\), let us now define \(K(\delta) = K \setminus B(\theta_0, \delta)\). \(K(\delta)\) is obviously a compact set since \(K\) itself is a compact set. By compacity, from the covering
\begin{align*}
\bigcup_{\theta^\prime\in K(\delta)} &B(\theta^\prime, r(\theta^\prime)) \supset K(\delta),
\intertext{there exists a finite subcovering, i.e.}
\exists m(\delta)\in\N, \; \bigcup_{j=1}^{m(\delta)} &B(\theta_j^\prime, r(\theta_j^\prime)) \supset K(\delta).
\end{align*}
In particular, \eqref{eq.sylwesterps4} holds for \(\theta^\prime=\theta_j^\prime, j=1,\ldots,m(\delta)\). Let us define
\begin{align*}
n_0(\delta) &= \max_{j=1,\ldots,m(\delta)} n(\theta_j^\prime).
\end{align*}
We may now write
\begin{align*}
&\forall \delta > 0, \;\exists n_0(\delta)\in\N,\; \exists m(\delta)\in\N,\; \forall j=1,\ldots,m(\delta), \\
&\qquad \proba\left(\forall n>n_0(\delta), \; l_{1:n}(X_{1:n}|B(\theta_j^\prime, r(\theta_j^\prime)))-l_{1:n}(X_{1:n}|K) < -1 \right) = 1,
\intertext{which we turn into}
&\forall \delta > 0, \;\exists n_0(\delta)\in\N,\; \exists m(\delta)\in\N,\; \\
&\qquad \proba\left(\forall n>n_0(\delta), \; \forall j=1,\ldots,m(\delta), \; l_{1:n}(X_{1:n}|B(\theta_j^\prime, r(\theta_j^\prime)))-l_{1:n}(X_{1:n}|K) < -1 \right) = 1,
\end{align*}
thanks to the finiteness of \(m(\delta)\), and finally into
\begin{align*}
\forall \delta > 0, \;\exists n_0(\delta)\in\N,\; \proba\left(\forall n>n_0(\delta), \; l_{1:n}(X_{1:n}|K(\delta))-l_{1:n}(X_{1:n}|K) < -1 \right) = 1,
\end{align*}
because of the covering
\begin{align*}
\bigcup_{j=1}^{m(\delta)} B(\theta_j^\prime, r(\theta_j^\prime)) &\supset K(\delta).
\end{align*}
Let us now sum up what we have obtained so far. We proved that
\begin{align*}
\forall \delta > 0, \;\exists n_0(\delta)\in\N, \; \proba\left(\text{if } \forall n>n_0(\delta), \; l_{1:n}(X_{1:n}|\theta)-l_{1:n}(X_{1:n}|K) \geqslant \log e^{-1}, \text{ then } \theta \not\in K(\delta) \right) = 1,
\intertext{i.e.}
\exists a = e^{-1} \in\interoo{0}{1}, \;\forall \delta > 0, \;\exists n_0(\delta)\in\N, \;\proba\left(\text{if } \forall n>n_0(\delta), \; \theta\in K_n(a), \text{ then } \|\theta-\theta_0\|_1 < \delta \right) = 1,
\end{align*}
that is to say
\begin{align*}
\exists a\in\interoo{0}{1}, \;\proba\left(\lim_{n\conv+\infty} \sup_{\theta\in K_n(a)} \|\theta-\theta_0\|_1 = 0 \right) &= 1.
\end{align*}
\end{proof}

\begin{proof}[\underline{Proof of Proposition \ref{prop.feder314}}]In this proof \(\|\cdot\|\) will refer to the usual Euclidean norm. Reindexing whenever necessary, we also assume that the observations \(t_i\) are ordered, and we denote
\begin{align*}
t&=(t_1,\ldots,t_n), & X&=(X_1, \ldots, X_n),& \mu_0 &= (\mu(\eta_0, t_1), \ldots, \mu(\eta_0, t_n)),
\end{align*}
\begin{align*}
N_{0,n} &= \sup_{i\leqslant n} \{ i, \; t_i < u_0 \} = \frac{1}{n}\sum_{i=1}^n \one_{\interfo{t_i}{+\infty}}(u_0), & N_n &= \sup_{i\leqslant n} \{ i, \; t_i < \widehat{u}_n \} = \frac{1}{n}\sum_{i=1}^n \one_{\interfo{t_i}{+\infty}}(\widehat{u}_n),
\end{align*}
\begin{align*}
\zeta = \left\{\begin{array}{ll}(0,\ldots,0, \beta_0 + \gamma_0 t_{N_n+1}, \ldots, \beta_0 + \gamma_0 t_{N_{0,n}}, 0, \ldots, 0), & \text{if } N_n < N_{0,n} \\ (0,\ldots,0), & \text{if } N_n = N_{0,n} \\ (0,\ldots,0, \beta_0 + \gamma_0 t_{N_{0,n}+1}, \ldots, \beta_0 + \gamma_0 t_{N_n}, 0, \ldots, 0), & \text{if } N_n > N_{0,n}\end{array}\right.,
\end{align*}

Let \(\mathcal{G}\) be the linear space spanned by the 2 linearly independent \(n\)-vectors
\begin{align*}
v_1 &= (1, \ldots, 1, 0, \ldots, 0) & v_2 &= (t_{1}, \ldots, t_{N_n}, 0, \ldots, 0)
\end{align*}
(both of which have their last \(n-N_n\) coordinates valued to zero), and denote \(Q\) the orthogonal projection onto \(\mathcal{G}\).

Let \(\mathcal{G}^+\) denote the linear space spanned by \(v_1\), \(v_2\) and \(\mu_0\) and denote \(Q^+\) the orthogonal projection onto \(\mathcal{G}^+\). Observe that \(\mathcal{G}^+\) is also spanned by \(v_1\), \(v_2\) and \(\zeta\).

Finally, denote \(\mu^*\) the orthogonal projection of \(X\) onto \(\mathcal{G}^+\) and \(\widehat{\mu}\) the closest point to \(X\) in \(\mathcal{G}^+\) satisfying the continuity assumption of the model, i.e.
\begin{align*}
 \mu^* &= Q^+ X, & \widehat{\mu} &= (\mu(\widehat{\eta}_n, t_1), \ldots, \mu(\widehat{\eta}_n, t_n)).
\end{align*}

We have
\begin{align*}
\|X-\mu^*\|^2 + \|\mu^*-\widehat{\mu}||^2 = \|X-\widehat{\mu}\|^2 &\leqslant \|X-\mu_0\|^2,\\
\|X-\mu_0\|^2 - \|\mu^*-\mu_0\|^2 + \|\mu^*-\widehat{\mu}||^2 &\leqslant \|X-\mu_0\|^2,\\
\|\mu^*-\mu_0\|^2 - 2\left<\mu^*-\mu_0,\widehat{\mu}-\mu_0\right> +\|\widehat{\mu}-\mu_0\|^2 &\leqslant \|\mu^*-\mu_0\|^2 .
\end{align*}
Thus
\begin{align*}
\|\widehat{\mu}-\mu_0\|^2 &\leqslant 2\left<\mu^*-\mu_0,\widehat{\mu}-\mu_0\right> \leqslant 2\|\mu^*-\mu_0\| \cdot \|\widehat{\mu}-\mu_0\|,
\intertext{which leads to}
\|\widehat{\mu}-\mu_0\| &\leqslant 2\|\mu^*-\mu_0\| \leqslant 2 \|Q^+ \xi\|.
\end{align*}

Our aim is to show that a.s.
\begin{align}
\|Q^+\xi\| &= \go\left(\log n\right). \label{eq.objectifQplusxi}
\end{align}
If \eqref{eq.objectifQplusxi} held, then we would have a.s. \(\|\widehat{\mu}-\mu_0\| = \go\left(\log n\right)\) i.e. a.s.
\begin{align*}
\sum_{i=1}^n \left(\mu(\widehat{\eta}_n, t_i - \mu(\eta_0, t_i)\right)^2 &= \go\left(\log^2 n\right).
\end{align*}
Hence, a.s. for any open interval \(I\subset \interff{\underline{u}}{\overline{u}}\) we would have
\begin{align*}
\sum_{i=1}^n \left(\mu(\widehat{\eta}_n, t_i - \mu(\eta_0, t_i)\right)^2 \one_{I}(t_i) &= \go\left(\log^2 n\right).
\end{align*}
This would immediately imply the desired result, i.e. that a.s.
\begin{align*}
\min_{t_i\in I,\; i\leqslant n} \left|\mu(\widehat{\eta}_n, t_i) - \mu(\eta_0, t_i)\right| = \go\left(n^{-\frac{1}{2}}\log n\right),
\end{align*}
since a.s.
\begin{align*}
\go\left(\log^2 n\right) &= \sum_{i=1}^n \left(\mu(\widehat{\eta}_n, t_i - \mu(\eta_0, t_i)\right)^2 \one_{I}(t_i) \geqslant n \cdot \min_{t_i\in I,\; i\leqslant n} \left|\mu(\widehat{\eta}_n, t_i) - \mu(\eta_0, t_i)\right|^2 \cdot \frac{1}{n}\sum_{i=1}^{n} \one_{I}(t_i),
\end{align*}
where (see Assumption (A1))
\begin{align*}
\frac{1}{n}\sum_{i=1}^{n} \one_{I}(t_i) = \int_I \ud F_n(t) \conv \int_I \ud F(t) = \int_I f(t)\ud t > 0.
\end{align*}

Let us now prove that \eqref{eq.objectifQplusxi} indeed holds. We consider the two following mutually exclusive situations.
\vspace{0.25em}\par\noindent\textsc{Situation A: \(\zeta = (0, \ldots, 0)\).~} In this situation
\begin{align}
\|Q^+\xi\|&=\|Q\xi\|, \label{eq.qplusegalq}
\end{align}
and Cochran's theorem guarantees that \(\|Q\xi\|^2 \sim \chi^2(2)\) for \(n\geqslant2\). Hence, via Corollary \ref{lemma.suiteiidOlogn}, a.s.
\begin{align}
\|Q\xi\|   &= \go\left(\log n\right),\label{eq.normeQxi}
\end{align}
and \eqref{eq.objectifQplusxi} follows from \eqref{eq.qplusegalq} and \eqref{eq.normeQxi}.
\vspace{0.25em}\par\noindent\textsc{Situation B: \(\zeta \neq (0, \ldots, 0)\).~} Since
\begin{align}
\frac{\left|\left<\zeta,\xi\right>\right|}{\|\zeta\|} &\sim \loi{N}(0, \sigma_0^2),\nonumber
\intertext{we also have, via Lemma \ref{lemma.suiteiidOlogn}, a.s.}
\frac{\left|\left<\zeta,\xi\right>\right|}{\|\zeta\|} &= \go\left(\log n\right). \label{eq.prodscalzetaxi}
\end{align}
Notice that \eqref{eq.objectifQplusxi} follows from \eqref{eq.normeQxi} and \eqref{eq.prodscalzetaxi} if we manage to show that a.s.
\begin{align}
\|Q^+\xi\| &\leqslant \go(1) \cdot \left(\|Q\xi\| + \frac{\left|\left<\zeta,\xi\right>\right|}{\|\zeta\|}\right). \label{eq.normeQplusxi}
\end{align}
It thus now suffices to prove that a.s., for any \(g\in \mathcal{G}\)
\begin{align}
\left|\left<\zeta, g\right>\right|&= \|\zeta\| \; \|g\| \cdot \po(1), \label{eq.prodscalzetagunifeng}
\end{align}
where the \(\po(1)\) mentioned in \eqref{eq.prodscalzetagunifeng} is uniform in \(g\) over \(\mathcal{G}\) (i.e. a.s. \(\zeta\) is asymptotically uniformly orthogonal to \(\mathcal{G}\)), for \eqref{eq.normeQplusxi} is a direct consequence of \eqref{eq.prodscalzetagunifeng} and Lemma \ref{lemma.controlalpha} whose proof is found in \cite{Feder}.

\begin{lemma}\label{lemma.controlalpha}Let \(\mathcal{X}\) and \(\mathcal{Y}\) be two linear subspaces of an inner product space \(\mathcal{E}\). If there exists \(\alpha < 1\) such that
\begin{align*}
\forall (x,y) \in\mathcal{X}\times\mathcal{Y}, \; \left|\left<x,y\right>\right| \leqslant \alpha \|x\|\; \|y\|,
\end{align*}
then
\begin{align*}
\|x + y\| \leqslant (1 - \alpha)^{-1} (\|x^*\| + \|y^*\|),
\end{align*}
where \(x^*\) (resp. \(y^*\)) is the orthogonal projection of \(x+y\) onto \(\mathcal{X}\) (resp. \(\mathcal{Y}\)).
\end{lemma}


Observe that, as a consequence of Assumption (A1) and Theorem \ref{theo.Polya}, the three following convergences are uniform in \(u\) over \(\interff{\underline{u}}{\overline{u}}\) for \(k=0,1,2\),
\begin{align}
\frac{1}{n}\sum_{i=1}^n t_i^k \one_{\interfo{t_i}{+\infty}}(u)
= \int_{\underline{u}}^{u} t^k \ud F_n(t)
&\conv \int_{\underline{u}}^{u} t^k \ud F(t)
= \int_{\underline{u}}^{u} t^k f(t) \ud t. \label{eq.3convunif}
\end{align}

We have a.s., for any \(g(\phi) = (\cos \phi) v_1 + (\sin \phi) v_2 \in \mathcal{G}\), with \(\phi \in \interff{0}{2\pi}\)
\begin{align*}
\left|\left<\zeta, g(\phi)\right>\right| &= \left|\sum_{i=1}^{N_n} (\beta_0 + \gamma_0 t_i) (\cos\phi + t_i \sin\phi) - \sum_{i=1}^{N_{0,n}} (\beta_0 + \gamma_0 t_i) (\cos\phi + t_i \sin\phi) \right |\\
&\leqslant (\max(|\underline{u}|, |\overline{u}|) + 1) \cdot \left|\sum_{i=1}^{N_n} \left|\beta_0 + \gamma_0 t_i\right| - \sum_{i=1}^{N_{0,n}} \left|\beta_0 + \gamma_0 t_i\right| \right| \\
&\leqslant (\max(|\underline{u}|, |\overline{u}|) + 1) \cdot \|\zeta\|_1 \\
&\leqslant (\max(|\underline{u}|, |\overline{u}|) + 1) \cdot \|\zeta\| \cdot n^{\frac{1}{2}} |N_n - N_{0,n}|^{\frac{1}{2}} \\
&\leqslant (\max(|\underline{u}|, |\overline{u}|) + 1) \cdot \|\zeta\| \cdot n^{\frac{1}{2}}\left|\frac{1}{n}\sum_{i=1}^n \one_{\interfo{t_i}{+\infty}}(\widehat{u}_n) - \frac{1}{n}\sum_{i=1}^n \one_{\interfo{t_i}{+\infty}}(u_0) \right|^{\frac{1}{2}},
\end{align*}
i.e. we have a.s. for any \(\phi \in \interff{0}{2\pi}\)
\begin{align}
\left|\left<\zeta, g(\phi)\right>\right| &= n^{\frac{1}{2}} \|\zeta\| \cdot \po(1), \label{eq.prodscalzetag}
\end{align}
thanks to the strong consistency \(\widehat{u}_n\convps u_0\) (see Theorem \ref{theo.strongconsistency}) and the uniform convergence mentioned in \eqref{eq.3convunif} with (\(k=0\)). Observe that the \(\po(1)\) mentioned in \eqref{eq.prodscalzetag} is uniform in \(\phi\) over \(\interff{0}{2\pi}\). We also have a.s. for any \(\phi \in \interff{0}{2\pi}\)
\begin{align*}
\frac{1}{n}\|g(\phi)\|^2 &= \frac{1}{n}\sum_{i=1}^n (\cos \phi + t_i \sin \phi)^2 \one_{\interfo{t_i}{+\infty}}(\widehat{u}_n) \\
&= \frac{1}{n}\sum_{i=1}^n \one_{\interfo{t_i}{+\infty}}(\widehat{u}_n) \cos^2 \phi + 2 \frac{1}{n}\sum_{i=1}^n t_i \one_{\interfo{t_i}{+\infty}}(\widehat{u}_n) \cos\phi \sin\phi\\
&\qquad + \frac{1}{n}\sum_{i=1}^n t_i^2 \one_{\interfo{t_i}{+\infty}}(\widehat{u}_n) \sin^2 \phi \\
&\convps \cos^2\phi \int_{\underline{u}}^{u_0} f(t) \ud t + \cos\phi \sin\phi \int_{\underline{u}}^{u_0} 2tf(t) \ud t +  \sin^2\phi \int_{\underline{u}}^{u_0} t^2f(t) \ud t,
\end{align*}
once again making use of the strong consistency \(\widehat{u}_n\convps u_0\) (see Theorem \ref{theo.strongconsistency}) and taking advantage of all three uniform convergences mentioned in \eqref{eq.3convunif}.
We thus obviously have a.s., uniformly in \(\phi\) over \(\interff{0}{2\pi}\)
\begin{align}
\frac{1}{n}\|g(\phi)\|^2 &\conv \int_{\underline{u}}^{u_0} (\cos\phi + t\sin\phi)^2 f(t) \ud t. \label{eq.convnormg2}
\end{align}
The limit in \eqref{eq.convnormg2} is a positive and continuous function of \(\phi\), and is hence bounded, i.e. there exists \(m>0\) such that we have a.s.
\begin{align}
\frac{1}{n}\|g(\phi)\|^2 &\geqslant m + \po(1), \label{eq.convnormg2bis}
\intertext{i.e.}
\frac{1}{\|g(\phi)\|} &= \go(n^{-\frac{1}{2}}), \label{eq.convnormg2ter}
\end{align}
where the \(\po(1)\) mentioned in \eqref{eq.convnormg2bis} and the \(\go(n^{-\frac{1}{2}})\) mentioned in \eqref{eq.convnormg2ter} are uniform in \(\phi\) over \(\interff{0}{2\pi}\).

Combining \eqref{eq.prodscalzetag} and \eqref{eq.convnormg2ter} together, we have a.s. for any \(\phi \in \interff{0}{2\pi}\)
\begin{align}
\left|\left<\zeta, g(\phi)\right>\right| &= \|\zeta\| \; \|g(\phi)\| \cdot \po(1), \label{eq.prodscalzetagunifenphi}
\end{align}
where the \(\po(1)\) mentioned in \eqref{eq.prodscalzetagunifenphi} is uniform in \(\phi\) over \(\interff{0}{2\pi}\).

Hence, we have a.s, for any \(r\in \R_+^*\), and any \(\phi \in \interff{0}{2\pi}\), now denoting \(g(\phi) = (r\cos\phi)v_1 + (r\sin\phi)v_2 \) and applying \eqref{eq.prodscalzetagunifenphi} to \(r^{-1}g(\phi)\)
\begin{align*}
\left|\left<\zeta, g(\phi)\right>\right| = r \left|\left<\zeta, r^{-1}g(\phi)\right>\right| &= r \cdot \|\zeta\| \; \|r^{-1}g(\phi)\| \cdot \po(1) = \|\zeta\| \; \|g(\phi)\| \cdot \po(1),
\end{align*}
where the \(\po(1)\) mentioned is uniform in \(\phi\) over \(\interff{0}{2\pi}\) and does not depend on \(r\).

We immediately deduce that a.s. \eqref{eq.prodscalzetagunifeng} holds i.e. a.s. \(\zeta\) is asymptotically uniformly orthogonal to \(\mathcal{G}\), which completes the proof.
\end{proof}

\subsection{Proofs of Section \ref{sec:asympdistMLE}}\label{subsec:proof4}
\begin{proof}[\underline{Proof of Proposition \ref{prop.asympdistpseudo}}]We proceed as announced.
\vspace{0.25em}\par\noindent\textsc{Step 1.~} We first prove that a.s.
  \begin{align*}
  \exists N\in\N, \; \forall n>N, \; \widehat{u}_n^* &\in D_n.
  \end{align*}
  Let us notice that anything proven for the problem remains valid for the pseudo-problem. Because \(n^*\sim n\), we have a.s., thanks to Theorem \ref{theo.strongconsistencywithrate} and conditions \eqref{eq.definitiondn}, as \(n\conv+\infty\)
  \begin{align*}
  n^{\frac{1}{2}} (\log^{-1} n) \cdot \left(\widehat{u}_n^* - u_0\right) &= \go(1), \\
  n^{\frac{1}{2}} (\log^{-1} n) \cdot d_n &\conv +\infty,
  \intertext{and thus deduce from the ratio of these two quantities that}
  \frac{\widehat{u}_n^* - u_0}{d_n} &\convps 0,
  \end{align*}
  and this directly implies the desired result.

\vspace{0.25em}\par\noindent\textsc{Step 2.~} Let \(A_{1:n}^*(\cdot)\) be the column vector defined for \(u\in D_n\) by
\begin{align}\label{eq.definitionAnstar}
A_{1:n}^*(\theta) = \left(\left.\frac{\partial l_{1:n}^*(X_{1:n}|\theta)}{\partial \gamma}\right|_{\theta}, \left.\frac{\partial l_{1:n}^*(X_{1:n}|\theta)}{\partial u}\right|_{\theta}, \left.\frac{\partial l_{1:n}^*(X_{1:n}|\theta)}{\partial \sigma^2}\right|_{\theta}\right).
\end{align}

  Step 1 allows us to expand a.s. \(A_{1:n}^*(\widehat{\theta}_n^*)\) around \(\theta_0\) using a Taylor-Lagrange approximation
  \begin{align*}
  0 &= A_{1:n}^*(\widehat{\theta}_n^*) = A_{1:n}^*(\theta_0) - B_{1:n}^*(\widetilde{\theta}_n) \left({\widehat\theta}_n^* - \theta_0\right),
  \end{align*}
  where \(\widetilde{\theta}_n\) is a point between \(\widehat{\theta}_n^*\) and \(\theta_0\) (see \eqref{eq.definitionBnstar} for the definitions of \(B_{1:n}^*\)), and rewrite it as a.s.
  \begin{align*}
  \frac{1}{n^*} B_{1:n}^*(\widetilde{\theta}_n) \cdot n^{*\frac{1}{2}}\left({\widehat\theta}_n^* - \theta_0\right) = n^{*-\frac{1}{2}}A_{1:n}^*(\theta_0).
  \end{align*}
  Since \(\widehat{\theta}_n^* \conv \theta_0\), we also have \(\widetilde{\theta}_n \conv \theta_0\) and using both Lemmas \ref{lemma.An} and \ref{lemma.Bn} we immediately find that as \(n\conv+\infty\)
  \begin{align*}
  I(\theta_0) \cdot n^{*\frac{1}{2}}\left({\widehat\theta}_n^* - \theta_0\right) &\convloi \loi{N}\left(0, I(\theta_0)\right),
  \intertext{which means, remembering both that \(n^{*} \sim n\) and that \(I(\theta_0)\) is positive definite and thus invertible that as \(n\conv+\infty\)}
  n^{\frac{1}{2}}\left({\widehat\theta}_n^* - \theta_0\right) &\convloi \loi{N}\left(0, I(\theta_0)^{-1}\right).
  \end{align*}
\end{proof}

\begin{proof}[\underline{Proof of Theorem \ref{theo.strongconsistencywithrate}}]
We now prove that \[\|\widehat{\sigma}_n^2-\sigma_0^2\| = \go\left(n^{-\frac{1}{2}}\log n\right).\]
\vspace{0.25em}\par\noindent\textsc{Variance of noise \(\sigma^2\).~} Observe that
\begin{align}
\widehat{\sigma}_n^2 &= \frac{1}{n} \sum_{i=1}^n \left[X_i - {\widehat\gamma}_n \left(t_i - \widehat{u}_n) \one_{\interfo{t_i}{+\infty}}(\widehat{u}_n\right) \right]^2 \nonumber \\
&= \frac{1}{n} \sum_{i=1}^n \left[\gamma_0\cdot(t_i - u_0) \one_{\interfo{t_i}{+\infty}}(u_0) - {\widehat\gamma}_n\cdot(t_i - \widehat{u}_n) \one_{\interfo{t_i}{+\infty}}(\widehat{u}_n) + \xi_i \right]^2 \nonumber \\
&= \frac{1}{n} \sum_{i=1}^n \nu_i^2(\widehat{\eta}_n) + \frac{2}{n} \sum_{i=1}^n \nu_i(\widehat{\eta}_n) \xi_i + \frac{1}{n} \sum_{i=1}^n \xi_i^2, \label{eq.sigma2ai}
\end{align}
where we denote for \(i=1,\ldots,n\),
\begin{align}
\nu_i(\eta) = \gamma_0\cdot(t_i - u_0) \one_{\interfo{t_i}{+\infty}}(u_0) - \gamma\cdot(t_i - u) \one_{\interfo{t_i}{+\infty}}(u). \label{eq.definitionsigma2ai}
\end{align}
We have
\begin{align}
\sup_{i\in\N} \left|\nu_i(\widehat{\eta}_n)\right| &= \sup_{i\in\N} \left|\gamma_0\cdot(t_i - u_0) \one_{\interfo{t_i}{+\infty}}(u_0) - {\widehat\gamma}_n\cdot(t_i - \widehat{u}_n) \one_{\interfo{t_i}{+\infty}}(\widehat{u}_n)\right| \nonumber \\
&\leqslant \left|\gamma_0 - \widehat{\gamma}_n\right| \cdot \sup_{i\in\N} \left|(t_i - u_0) \one_{\interfo{t_i}{+\infty}}(u_0)\right| \nonumber \\
&\qquad + \left|{\widehat\gamma}_n\right| \cdot \sup_{i\in\N} \left|(t_i - u_0) \one_{\interfo{t_i}{+\infty}}(u_0)-(t_i - \widehat{u}_n) \one_{\interfo{t_i}{+\infty}}(\widehat{u}_n)\right| \nonumber \\
&= \go\left(\gamma_0 - \widehat{\gamma}_n\right) + \left|{\widehat\gamma}_n\right| \go\left(u_0 - \widehat{u}_n\right), \label{eq.dominationai}
\end{align}
using straightforward dominations and Lemma \ref{lemma.convpuissk}, so that in the end, thanks to the previous results we have a.s.
\begin{align}
\sup_{i\in\N} \left|\nu_i(\widehat{\eta}_n)\right| &= \go\left(n^{-\frac{1}{2}}\log n\right).
\end{align}
It is thus easy to see that a.s.
\begin{align}
\frac{1}{n} \sum_{i=1}^n \nu_i^2(\widehat{\eta}_n) &= \go\left(n^{-1}\log^2 n\right)
= \go\left(n^{-\frac{1}{2}}\log n\right), \label{eq.moyenneaicarregrando}
\intertext{and also that, via Corollary \ref{lemma.suiteiidOlogn}, a.s.}
\frac{2}{n} \sum_{i=1}^n \nu_i(\widehat{\eta}_n) \xi_i &= \frac{2}{n} \left(\sum_{i=1}^n \nu_i^2(\widehat{\eta}_n)\right)^{\frac{1}{2}} \cdot \go(\log n)
= \go\left(n^{-\frac{1}{2}}\log n\right). \label{eq.moyenneaixiigrando}
\end{align}
From the Law of the Iterated Logarithm \cite[see][Chapter 13, page 291]{Breiman} we have a.s.
\begin{align}
\frac{1}{n} \sum_{i=1}^n \left(\xi_i^2 - \sigma_0^2\right) &= \go\left(n^{-\frac{1}{2}}(\log\log n)^{\frac{1}{2}}\right)
=\go\left(n^{-\frac{1}{2}}\log n\right) \label{eq.moyennexiicarregrando}
\end{align}
and the desired result follows from \eqref{eq.moyenneaicarregrando}, \eqref{eq.moyenneaixiigrando} and \eqref{eq.moyennexiicarregrando} put together into \eqref{eq.sigma2ai}.
\end{proof}

\begin{proof}[\underline{Proof of Theorem \ref{theo.asympdist}}]
To finish the proof, we need to show \eqref{eq.asymppseudosigma2} i.e. that
\begin{align*}
\widehat{\sigma}_n^2 - \widehat{\sigma}_n^{2*} &= \pop\left(n^{-\frac{1}{2}}\right).
\end{align*}
We use the decomposition \eqref{eq.sigma2ai}
\begin{align*}
\widehat{\sigma}_n^2 &= \frac{1}{n} \sum_{i=1}^n \nu_i^2(\widehat{\eta}_n) + \frac{2}{n} \sum_{i=1}^n \nu_i(\widehat{\eta}_n) \xi_i + \frac{1}{n} \sum_{i=1}^n \xi_i^2,
\end{align*}
where \(\nu_i(\widehat{\eta}_n) = \gamma_0 \cdot (t_i - u_0) \one_{\interfo{t_i}{+\infty}}(u_0) - \widehat{\gamma}_n \cdot (t_i - \widehat{u}_n) \one_{\interfo{t_i}{+\infty}}(\widehat{u}_n)\).

Having proved in Proposition \ref{prop.asympdistpseudo} that
\begin{align}
{\widehat\gamma}_n^* - \gamma_0 &= \gop\left(n^{-\frac{1}{2}}\right), \nonumber
&{\widehat u}_n^* - u_0 &= \gop\left(n^{-\frac{1}{2}}\right)
\intertext{we add these relationships to those from \eqref{eq.asymppseudogammau} and find that}
{\widehat\gamma}_n - \gamma_0 &= \gop\left(n^{-\frac{1}{2}}\right),
&{\widehat u}_n - u_0 &= \gop\left(n^{-\frac{1}{2}}\right). \label{eq.asympdistpseudogammau}
\end{align}

We now use \eqref{eq.asympdistpseudogammau} together with \eqref{eq.dominationai}, we are able to write
\begin{align}
\sup_{i\in\N} \left|\nu_i(\widehat{\eta}_n)\right| &= \gop\left(n^{-\frac{1}{2}}\right).
\end{align}
It is hence easy to see that
\begin{align*}
\frac{1}{n} \sum_{i=1}^n \nu_i^2(\widehat{\eta}_n) &= \gop\left(n^{-1}\right)
= \pop\left(n^{-\frac{1}{2}}\right),
\intertext{and also that}
\frac{2}{n} \sum_{i=1}^n \nu_i(\widehat{\eta}_n) \xi_i &= \frac{2}{n} \left(\sum_{i=1}^n \nu_i^2(\widehat{\eta}_n)\right)^{\frac{1}{2}} \cdot \gop(1)
= \pop\left(n^{-\frac{1}{2}}\right),
\end{align*}
which once both substituted into \eqref{eq.sigma2ai} yield
\begin{align*}
\widehat{\sigma}_n^2 &= \frac{1}{n}\sum_{i=1}^n \xi_i^2 + \pop\left(n^{-\frac{1}{2}}\right).
\end{align*}
What was done above with the problem and \(\widehat{\sigma}_n^2\) can be done with the pseudo-problem and \(\widehat{\sigma}_n^{2*}\) without any kind of modification so that
\begin{align*}
\widehat{\sigma}_n^{2*} &= \frac{1}{n^*}\sum_{i=1}^{n^*} \xi_i^2 + \pop\left(n^{-\frac{1}{2}}\right).
\end{align*}
We observe that
\begin{align*}
\widehat{\sigma}_n^2 - \widehat{\sigma}_n^{2*} &= \frac{1}{n}\sum_{i=1}^{n} \xi_i^2 - \frac{1}{n^*}\sum_{i=1}^{n^*} \xi_i^2 + \pop\left(n^{-\frac{1}{2}}\right) \\
&= \left[\frac{1}{n}-\frac{1}{n^*}\right] \cdot \sum_{i=1}^{n^*} \xi_i^2 + \frac{1}{n} \cdot \sum_{i=n^*+1}^{n} \xi_i^2  + \pop\left(n^{-\frac{1}{2}}\right) \\
&= \frac{n^*-n}{n} \cdot \left(\frac{1}{n^*} \sum_{i=1}^{n^*} \xi_i^2\right) + \frac{n-n^*}{n} \cdot \left(\frac{1}{n-n^*} \sum_{i=n^*+1}^{n} \xi_i^2\right)  + \pop\left(n^{-\frac{1}{2}}\right) \\
&= \frac{n^*-n}{n} \cdot \left(\sigma_0^2 + \gop\left(n^{*-\frac{1}{2}}\right)\right) + \frac{n-n^*}{n} \cdot \left(\sigma_0^2 + \gop\left(\left(n-n^*\right)^{-\frac{1}{2}}\right)\right)  + \pop\left(n^{-\frac{1}{2}}\right),
\intertext{using the Central Limit Theorem, and in the end we get}
\widehat{\sigma}_n^2 - \widehat{\sigma}_n^{2*}
&= \frac{n^*-n}{n} \cdot \gop\left(n^{*-\frac{1}{2}}\right) + \frac{n-n^*}{n} \cdot \gop\left(\left(n-n^*\right)^{-\frac{1}{2}}\right) + \pop\left(n^{-\frac{1}{2}}\right) \\
&= \po(1) \cdot \gop\left(n^{-\frac{1}{2}}\right) + n^{-\frac{1}{2}} \cdot \gop\left(\left(\frac{n-n^*}{n}\right)^{\frac{1}{2}}\right) + \pop\left(n^{-\frac{1}{2}}\right) \\
&= \pop\left(n^{-\frac{1}{2}}\right) + n^{-\frac{1}{2}} \cdot \gop\left(\po(1)\right)  + \pop\left(n^{-\frac{1}{2}}\right)
= \pop\left(n^{-\frac{1}{2}}\right).
\end{align*}
\end{proof}

\section{Technical results}\label{sec:prooftechnical}
\begin{theorem}[Polya's Theorem]\label{theo.Polya}Let \((g_n)_{n\in\N}\) be a sequence of non decreasing (or non increasing) functions defined over \(I = \interff{a}{b}\subset\R\). If \(g_n\) converges pointwise to \(g\) (i.e. \(g_n(x) \conv g(x)\) as \(n\conv+\infty\), for any \(x\in I\)) and \(g\) is continuous then
\begin{align*}
\sup_{x\in I} \left|g_n(x) - g(x)\right| \xrightarrow[n\conv+\infty]{} 0.
\end{align*}
\end{theorem}

\begin{proof}[Proof of Lemma \ref{theo.Polya}]
Assume the functions \(g_n\) are non decreasing over \(I\) (if not, consider their opposites \(-g_n\)). \(g\) is continuous over \(I\) and thus bounded since \(I\) is compact. \(g\) is also non decreasing over \(I\) as the limit of a sequence of non decreasing functions. Let \(\epsilon > 0\) and \(k> \frac{g(b)-g(a)}{\epsilon}\) such that
\begin{align*}
\exists a=a_0 < \ldots < a_k=b \in I^{k+1}, \;\forall i=0,\ldots,k-1, \; g(a_{i+1})-g(a_i) < \epsilon.
\end{align*}
Now let \(x \in I\) and let \(i\in\N\) such that \(a_i \leqslant x \leqslant a_{i+1}\). Since \(g_n\) and \(g\) are non decreasing, we find that
\begin{align*}
g_n(x) - g(x) \leqslant g_n(a_{i+1})-g(a_i) &\leqslant g_n(a_{i+1}) - g(a_{i+1}) + \epsilon,\\
g_n(x) - g(x) \geqslant g_n(a_i)-g(a_{i+1}) &\geqslant g_n(a_i) - g(a_i) - \epsilon.
\end{align*}
The pointwise convergence of \(g_n\) to \(g\) and the finiteness of \(k\) together ensure that
\begin{align*}
\exists N_0\in\N, \;\forall n\geqslant N_0, \;\forall i=0,\ldots,k, \;\left|g_n(a_i)-g(a_i)\right| < \epsilon,
\end{align*}
which implies with both of the inequations mentioned above that
\begin{align*}
\exists N_0\in\N, \;\forall n\geqslant N_0, \;\forall x\in I, \;\left|g_n(x)-g(x)\right| < \epsilon.
\end{align*}
\end{proof}

\begin{lemma}\label{lemma.convpuissk}
Let \(k\in\N^*\), there exists a constant \(C\in\R_+^*\) such that for any \((u, u^\prime) \in \interff{\underline{u}}{\overline{u}}^2\)
\begin{align}
\sup_{t\in \interff{\underline{u}}{\overline{u}}} | (t - u^\prime)^k \one_{\interfo{t}{+\infty}}(u^\prime) - (t - u)^k \one_{\interfo{t}{+\infty}}(u) | &=  C|u-u^\prime|. \label{eq.lemma.convpuissk1}
\end{align}
\end{lemma}

\begin{proof}[Proof of Lemma \ref{lemma.convpuissk}]
For any \((u, u^\prime) \in \interff{\underline{u}}{\overline{u}}^2\) we have
\begin{align}
\sup_{t\in \interff{\underline{u}}{\overline{u}}} | (t - u^\prime)^k \one_{\interfo{t}{+\infty}}(u^\prime) -  (t - u)^k \one_{\interfo{t}{+\infty}}(u)|
&\leqslant  \sup_{t\in \interff{\underline{u}}{\overline{u}}} \{ |(t - u^\prime)^k - (t - u)^k| \one_{\interfo{t}{+\infty}}(u^\prime) \} \nonumber\\
&\quad + \sup_{t\in \interff{\underline{u}}{\overline{u}}} \{ |t - u|^k |\one_{\interfo{t}{+\infty}}(u^\prime) - \one_{\interfo{t}{+\infty}}(u) | \}. \label{eq.lemma.convpuissk2}
\end{align}
The mean value theorem guarantees that there exists \(v\) between \(u\) and \(u^\prime\) such that
\begin{align*}
(t-u^\prime)^k - (t-u)^k &= -k(t-v)^{k-1}(u^\prime-u).
\end{align*}
We thus have
\begin{align}
\sup_{t\in \interff{\underline{u}}{\overline{u}}} \{|(t-u^\prime)^k - (t-u)^k| \one_{\interfo{t}{+\infty}}(u^\prime)\} &\leqslant \sup_{t\in \interff{\underline{u}}{\overline{u}}} |(t-u^\prime)^k - (t-u)^k| \nonumber\\
&\leqslant k|\overline{u}-\underline{u}|^{k-1} |u-u^\prime|. \label{eq.lemma.convpuissk3}
\end{align}
Because \( |t-u| \leqslant |u^\prime-u| \) whenever \(|\one_{\interfo{t}{+\infty}}(u^\prime) - \one_{\interfo{t}{+\infty}}(u) | \neq 0\), we also find that
\begin{align}
\sup_{t\in \interff{\underline{u}}{\overline{u}}} \{|t - u|^k |\one_{\interfo{t}{+\infty}}(u^\prime) - \one_{\interfo{t}{+\infty}}(u) | \} \leqslant |u-u^\prime|^k \leqslant |u-u^\prime| |\overline{u}-\underline{u}|^{k-1}. \label{eq.lemma.convpuissk4}
\end{align}
And now \eqref{eq.lemma.convpuissk1} is a simple consequence of \eqref{eq.lemma.convpuissk2}, \eqref{eq.lemma.convpuissk3} and \eqref{eq.lemma.convpuissk4}.
\end{proof}

\begin{lemma}\label{lemma.convergenceuniformemu}
For any \(\eta^\prime\in\R\times\interff{\underline{u}}{\overline{u}}\), there exists \(C\in\R_+^*\) such that for any \(\eta\in\R\times\interff{\underline{u}}{\overline{u}}\)
\begin{align}
\sup_{t\in\interff{\underline{u}}{\overline{u}}}|\mu(\eta, t)-\mu(\eta^\prime, t)|
&\leqslant C\|\eta-\eta^\prime\|.\label{eq.lemma.convergenceuniformemu1}
\end{align}
\end{lemma}
\begin{proof}[Proof of Lemma \ref{lemma.convergenceuniformemu}]We have indeed
\begin{align*}
\sup_{t\in\interff{\underline{u}}{\overline{u}}}|\mu(\eta, t)-\mu(\eta^\prime, t)|
&= \sup_{t\in\interff{\underline{u}}{\overline{u}}}|\gamma\cdot(t-u)\one_{\interfo{t}{+\infty}}(u) - \gamma^\prime(t-u^\prime)\one_{\interfo{t}{+\infty}}(u^\prime)| \\
&\leqslant \sup_{t\in\interff{\underline{u}}{\overline{u}}}|[\gamma-\gamma^\prime](t-u)\one_{\interfo{t}{+\infty}}(u)| \\
&\qquad + \sup_{t\in\interff{\underline{u}}{\overline{u}}}|\gamma^\prime[(t-u)\one_{\interfo{t}{+\infty}}(u) - (t-u^\prime)\one_{\interfo{t}{+\infty}}(u^\prime)]| \\
&\leqslant |\gamma-\gamma^\prime| \cdot \sup_{t\in\interff{\underline{u}}{\overline{u}}}|t-u| \\
&\qquad + |\gamma^\prime| \cdot \sup_{t\in\interff{\underline{u}}{\overline{u}}} |(t-u)\one_{\interfo{t}{+\infty}}(u) - (t-u^\prime)\one_{\interfo{t}{+\infty}}(u^\prime)| \\
&\leqslant |\gamma-\gamma^\prime| \cdot |\overline{u}-\underline{u}| + |\gamma^\prime| \cdot \sup_{t\in\interff{\underline{u}}{\overline{u}}} |(t-u)\one_{\interfo{t}{+\infty}}(u) - (t-u^\prime)\one_{\interfo{t}{+\infty}}(u^\prime)|.
\end{align*}
And now \eqref{eq.lemma.convergenceuniformemu1} is a simple consequence of Lemma \ref{lemma.convpuissk}.
\end{proof}

\begin{lemma}\label{lemma.grid}
Let \(A\subset\R\times\interff{\underline{u}}{\overline{u}}\) be a bounded set. Then,
\begin{align*}
&\forall \epsilon>0,\; \exists m(\epsilon)\in\N,\; \exists \eta_1, \ldots, \eta_{m(\epsilon)} \in A,\\
&\forall\eta,\eta^\prime\in A,\; \exists j,j^\prime\in\{1,\ldots,m(\epsilon)\},\; \sup_{t\in\interff{\underline{u}}{\overline{u}}} \left|\big[\mu(\eta, t) - \mu(\eta^\prime, t)\big] - \big[\mu(\eta_j, t) - \mu(\eta_{j^\prime}, t)\big]\right| < \epsilon,
\end{align*}
\end{lemma}

\begin{proof}[Proof of Lemma \ref{lemma.grid}]It suffices to prove the following claim
\begin{align*}
&\forall \epsilon>0,\; \exists m(\epsilon)\in\N,\; \exists \eta_1, \ldots, \eta_{m(\epsilon)} \in A, \\
&\forall\eta\in A,\; \exists j\in\{1,\ldots,m(\epsilon)\},\; \sup_{t\in\interff{\underline{u}}{\overline{u}}} |\mu(\eta, t)-\mu(\eta_j, t)| < \epsilon.
\end{align*}
and then use the triangle inequality. To see that the claim holds, it suffices, thanks to Lemma \ref{lemma.convergenceuniformemu}, to exhibit a finite and tight enough grid of \(A\) such that any point in \(A\) lies close enough to a point of the grid. The existence of such a grid is obviously guaranteed since \(A\subset\R^2\) is bounded.
\end{proof}

\begin{lemma}\label{lemma.sylwesterbnandb}Recall the definition of \(b_n\) given in \eqref{eq.definitionsylwesterbn}. Let
\begin{align}
b(\theta) &= \left(\frac{\sigma_0^2}{\sigma^2}-1-\log\frac{\sigma_0^2}{\sigma^2}\right) + \frac{1}{\sigma^2}\int_{\underline{u}}^{\overline{u}}\left[\mu(\eta_0, t) - \mu(\eta, t)\right]^2 f(t)\ud t.\label{eq.definitionsylwesterb}
\end{align}
Then, under Assumptions (A1)--(A4),
\begin{align}
 &b_n(\theta) \geqslant 0. \label{eq.sylwesterbnandb1}\\
 &b(\theta) \geqslant 0, \text{ with equality if and only if }\theta = \theta_0. \label{eq.sylwesterbnandb2}\\
 &b_n(\theta^\prime) \conv b_n(\theta), \text{ uniformly in } n, \text{ as } \theta^\prime\conv\theta. \label{eq.sylwesterbnandb3} \\
 &b_n(\theta) \conv b(\theta), \text{ as } n\conv+\infty. \label{eq.sylwesterbnandb4}
\end{align}
\end{lemma}
\begin{proof}[Proof of Lemma \ref{lemma.sylwesterbnandb}]We will prove each claim separately.

\vspace{0.25em}\par\noindent\textsc{Proof of \eqref{eq.sylwesterbnandb1}.~} That \(b_n(\theta)\geqslant 0\) is trivial since the first term in \eqref{eq.definitionsylwesterbn} is non negative (having \(x-1-\log x \geqslant 0\) with equality only if \(x=1\)), and the second term in \eqref{eq.definitionsylwesterbn} is obviously non negative too.

\vspace{0.25em}\par\noindent\textsc{Proof of \eqref{eq.sylwesterbnandb2}.~} That \(b(\theta)\geqslant 0\) is again easy enough to prove, both terms in \eqref{eq.definitionsylwesterb} being trivially non negative. If \(\theta\neq\theta_0\) then either \(\sigma^2\neq\sigma_0^2\) which implies the first term is positive, or \(\mu(\eta_0, \cdot) \neq \mu(\eta, \cdot)\) which implies the second term is positive (since \(f\) is assumed positive on \(\interff{\underline{u}}{\overline{u}}\)). Hence if \(\theta\neq\theta_0\) then \(b(\theta) > 0\). That \(\theta=\theta_0\) implies \(b(\theta) = 0\) is of course straightforward.

\vspace{0.25em}\par\noindent\textsc{Proof of \eqref{eq.sylwesterbnandb3}.~} We first observe that
\begin{align}
&~\left|\frac{1}{n}\sum_{i=1}^{n}(\mu(\eta_0, t_i) - \mu(\eta^\prime, t_i))^2 - \frac{1}{n}\sum_{i=1}^{n}(\mu(\eta_0, t_i) - \mu(\eta, t_i))^2\right| \nonumber \\
&= \left|\frac{1}{n}\sum_{i=1}^{n}\left[2\mu(\eta_0, t_i) - \mu(\eta^\prime, t_i) - \mu(\eta, t_i)\right]\cdot\left[\mu(\eta, t_i) - \mu(\eta^\prime, t_i)\right] \right| \nonumber \\
&\leqslant \frac{1}{n}\sum_{i=1}^{n}\left|2\mu(\eta_0, t_i) - \mu(\eta^\prime, t_i) - \mu(\eta, t_i)\right|\cdot\left|\mu(\eta, t_i) - \mu(\eta^\prime, t_i)\right| \nonumber \\
&\leqslant \left(\sup_{t\in\interff{\underline{u}}{\overline{u}}}\left|\mu(\eta_0, t) - \mu(\eta, t)\right| + \sup_{t\in\interff{\underline{u}}{\overline{u}}} \left|\mu(\eta^\prime, t) - \mu(\eta, t)\right|\right) \cdot \sup_{t\in\interff{\underline{u}}{\overline{u}}} \left|\mu(\eta^\prime, t) - \mu(\eta, t)\right|. \label{ineq.bn}
\end{align}
As \(\theta^\prime\conv\theta\), the convergence of the first term of \(b_n\) to the first term of \(b\) is obviously uniform in \(n\) since this part of \(b_n\) does not involve \(n\) at all. As \(\theta^\prime\conv\theta\), via Lemma \ref{lemma.convergenceuniformemu}, we also obtain
\begin{align*}
\sup_{t\in\interff{\underline{u}}{\overline{u}}} \left|\mu(\eta^\prime, t) - \mu(\eta, t)\right|\conv 0,
\end{align*}
which ensures that the second part of \eqref{eq.definitionsylwesterbn} converges uniformly in \(n\) thanks to \eqref{ineq.bn}.

\vspace{0.25em}\par\noindent\textsc{Proof of \eqref{eq.sylwesterbnandb4}.~} Thanks to Assumption (A1), it is easy to see that
\begin{align*}
\frac{1}{n}\sum_{i=1}^{n}\left[\mu(\eta_0, t_i) - \mu(\eta, t_i)\right]^2
&= \int_{\underline{u}}^{\overline{u}}\left[\mu(\eta_0, t) - \mu(\eta, t)\right]^2 \ud F_n(t) \\
&\conv \int_{\underline{u}}^{\overline{u}}\left[\mu(\eta_0, t) - \mu(\eta, t)\right]^2 \ud F(t)
= \int_{\underline{u}}^{\overline{u}}\left[\mu(\eta_0, t) - \mu(\eta, t)\right]^2 f(t)\ud t.
\end{align*}
\end{proof}

\begin{lemma}\label{lemma.sylwestertheo35}
Let \(A\subset\R\times\interff{\underline{u}}{\overline{u}}\) be a bounded set, and let \(\eta_0\in A\), then under Assumptions (A1)--(A4),
\begin{align*}
\sup_{\eta\in A} \left|\frac{1}{n}\sum_{i=1}^n [\mu(\eta_0, t_i) - \mu(\eta, t_i)] \xi_i\right| &\convps 0.
\end{align*}
\end{lemma}

\begin{proof}[Proof of Lemma \ref{lemma.sylwestertheo35}]Let \(\epsilon >0\), \(\eta\in A\), and apply Lemma \ref{lemma.grid} to get the corresponding \(m(\epsilon)\in\N\), \(\{\eta_1, \ldots, \eta_{m(\epsilon)}\} \subset A\), \(j,j^\prime\in\{1,\ldots,m(\epsilon)\}\).  We can write with the triangle inequality
\begin{align*}
\frac{1}{n}\sum_{i=1}^n [\mu(\eta_0, t_i) - \mu(\eta, t_i)] \xi_i &= \frac{1}{n }\sum_{i=1}^n [\mu(\eta_j, t_i) - \mu(\eta_{j^\prime}, t_i)] \xi_i \\
&\qquad +\frac{1}{n}\sum_{i=1}^n \left\{[\mu(\eta_0, t_i) - \mu(\eta, t_i)] - [\mu(\eta_j, t_i) - \mu(\eta_{j^\prime}, t_i)]\right\} \xi_i \\
&\leqslant \sup_{(j,j^\prime)\in\{1,\ldots,m(\epsilon)\}} \left\{\frac{1}{n}\sum_{i=1}^n [\mu(\eta_j, t_i) - \mu(\eta_{j^\prime}, t_i)] \xi_i\right\} + \epsilon \cdot \frac{1}{n}\sum_{i=1}^n \left|\xi_i\right|.
\end{align*}
Hence
\begin{align}
\sup_{\eta\in A} \left\{\frac{1}{n}\sum_{i=1}^n [\mu(\eta_0, t_i) - \mu(\eta, t_i)] \xi_i\right\}  &\leqslant \sup_{(j,j^\prime)\in\{1,\ldots,m(\epsilon)\}} \left\{\frac{1}{n}\sum_{i=1}^n [\mu(\eta_j, t_i) - \mu(\eta_{j^\prime}, t_i)] \xi_i\right\} \nonumber \\
&\qquad + \epsilon \cdot \frac{1}{n}\sum_{i=1}^n \left|\xi_i\right|.\label{ineq.sylwestertheo35}
\end{align}
Let us now recall Kolmogorov's criterion, a proof of which is available in Section 17 of \cite{Loeve} on pages 250--251. This criterion guarantees that for any sequence \((Y_i)_{i\in\N}\) of independent random variables and any numerical sequence \((b_i)_{i\in\N}\) such that
\begin{align*}
\sum_{i=1}^{+\infty} \dfrac{\var Y_i}{b_i^2} < +\infty, \; b_n &\conv +\infty,
\end{align*}
we have
\begin{align*}
\frac{\sum_{i=1}^n (Y_i - \esp Y_i)}{b_n} &\convps 0.
\end{align*}
For each couple \((j,j^\prime)\in\{1,\ldots,m(\epsilon)\}\), Kolmogorov's criterion ensures that
\begin{align*}
\frac{1}{n}\sum_{i=1}^n [\mu(\eta_j, t_i) - \mu(\eta_{j^\prime}, t_i)] \xi_i &\convps 0,
\end{align*}
for the coefficients \([\mu(\eta_j, t_i) - \mu(\eta_{j^\prime}, t_i)]\) are obviously bounded, and it suffices to pick \(Y_i = [\mu(\eta_j, t_i) - \mu(\eta_{j^\prime}, t_i)] \xi_i\) and \(b_i=i\). Having only a finite number of couples \((j,j^\prime)\in\{1,\ldots,m(\epsilon)\}^2\) to consider allows us to write
\begin{align}
\sup_{(j,j^\prime)\in\{1,\ldots,m(\epsilon)\}} \frac{1}{n}\sum_{i=1}^n [\mu(\eta_j, t_i) - \mu(\eta_{j^\prime}, t_i)] \xi_i &\convps 0.\label{eq.sylwestertheo35stronglawsup}
\end{align}
By \eqref{eq.sylwestertheo35stronglawsup}, the first term on the right hand side of \eqref{ineq.sylwestertheo35} converges almost surely to zero. The Strong Law of Large Numbers ensures that the second term on the right hand side of \eqref{ineq.sylwestertheo35} converges almost surely to \(\epsilon\cdot(2\pi^{-1}\sigma^2)^{\frac{1}{2}}\), and the result follows, since all the work done above for \((\xi_n)_{n\in\N}\) can be done again for \((-\xi_n)_{n\in\N}\).
\end{proof}

\begin{lemma}\label{lemma.suiteiidOlogn}
Let \((Z_i)_{i\in\N}\) be a sequence of independent identically distributed random variables such that for all \(i\in\N\), either \(Z_i\sim\loi{N}(0, \sigma^2)\) with \(\sigma^2 > 0\), or \(Z_i\sim\chi^2(k)\) with \(k > 0\). Then a.s., as \(n\conv+\infty\)
\begin{align*}
Z_n = \go(\log n).
\end{align*}
\end{lemma}

\begin{proof}[Proof of Lemma \ref{lemma.suiteiidOlogn}]
Denote \(Y_n = Z_n\) when the random variables are Gaussian, and \(Y_n = Z_n/5\) when the random variables considered are chi-squared (so that \(\esp e^{2Y_1}\) and \(\esp e^{-2Y_1}\) are both finite). We will show that a.s. \(Y_n = \go(\log n)\).

For any \(\epsilon > 0\), from Markov's inequality we get:
\begin{align*}
\proba\left(n^{-1}|e^{Y_n}| > \epsilon\right) = \proba\left(n^{-2}e^{2Y_n} > \epsilon^2\right) \leqslant \epsilon^{-2} n^{-2} \esp e^{2Y_1}.
\end{align*}
From there it is easy to see that for any \(\epsilon > 0\) we have
\begin{align*}
\sum_{n=1}^{+\infty} \proba\left(n^{-1}|e^{Y_n}| > \epsilon\right) = \epsilon^{-2}\frac{\pi^2}{6} \esp e^{2Y_1}< \infty,
\end{align*}
which directly implies via Borel-Cantelli's Lemma \cite[see for example][Section 4, page 59]{BillingsleyPM} that a.s.
\begin{align*}
e^{Y_n} &= \po(n).
\end{align*}
In particular, a.s. for any \(n\) large enough,
\begin{align*}
Y_n &\leqslant \log n.
\end{align*}
What was done with \((Y_n)_{n\in\N}\) can be done again with \((-Y_n)_{n\in\N}\) so that in the end we have a.s for any \(n\) large enough,
\begin{align*}
-\log n \leqslant Y_n &\leqslant \log n.
\end{align*}
\end{proof}

\begin{lemma}\label{lemma.majorationetamoinseta0}
Under Assumptions (A1)--(A4), for any \(\eta_0\in\R\times\interff{\underline{u}}{\overline{u}}\), there exists \(C\in \R_+^*\) such that for any \(n\) large enough, and for any \(\eta\)
\begin{align*}
n^{-1}\sum_{i=1}^n \left[\mu(\eta_0, t_i)-\mu(\eta,t_i)\right]^2 &\geqslant C\|\eta-\eta_0\|^2.
\end{align*}
\end{lemma}
\begin{proof}[Proof of Lemma \ref{lemma.majorationetamoinseta0}]We have already almost proved this result in \eqref{eq.dominetaneta0} (see Theorem \ref{theo.strongconsistencywithrate}). There is however a small difficulty since the majoration was obtained for \(\tau = (\beta,\gamma)\) and not \(\eta = (\gamma, u)\).

Let \(V_1\) and \(V_2\) two non empty open intervals of \(\interoo{\underline{u}}{u_0}\) such that their closures \(\overline{V_1}\) and \(\overline{V_2}\) are do not overlap. We have
\begin{align*}
n^{-1}\sum_{i=1}^n \left[\mu(\eta_0, t_i)-\mu(\eta,t_i)\right]^2 \geqslant n^{-1} &\left( \sum_{i=1}^n \left[\mu(\eta_0, t_i)-\mu(\eta,t_i)\right]^2\one_{V_1}(t_i) +\right.\\
&\qquad \left. \sum_{i=1}^n \left[\mu(\eta_0, t_i)-\mu(\eta,t_i)\right]^2\one_{V_2}(t_i) \right).
\end{align*}
Using the same arguments we used to prove \eqref{eq.dominetaneta0}, we find that there exists \(C\in\R_+^*\) such that (remembering the definition of the intercept \(\beta\) of the model)
\begin{align*}
n^{-1}\sum_{i=1}^n \left[\mu(\eta_0, t_i)-\mu(\eta,t_i)\right]^2 &\geqslant \min\left(n^{-1}\sum_{i=1}^n \one_{V_1}(t_i), n^{-1}\sum_{i=1}^n \one_{V_2}(t_i)\right) \cdot C |\gamma-\gamma_0|^2, \\
n^{-1}\sum_{i=1}^n \left[\mu(\eta_0, t_i)-\mu(\eta,t_i)\right]^2 &\geqslant \min\left(n^{-1}\sum_{i=1}^n \one_{V_1}(t_i), n^{-1}\sum_{i=1}^n \one_{V_2}(t_i)\right) \cdot C |\beta-\beta_0|^2,
\end{align*}
and since for \(j=1,2\) we have
\begin{align*}
n^{-1}\sum_{i=1}^n \one_{V_j}(t_i) &\conv \int_{V_j} f(t)\ud t > 0,
\end{align*}
there exists \(C\in\R_+^*\) such that for any \(n\) large enough
\begin{align*}
n^{-1}\sum_{i=1}^n \left[\mu(\eta_0, t_i)-\mu(\eta,t_i)\right]^2 &\geqslant C |\gamma-\gamma_0|^2, \\
n^{-1}\sum_{i=1}^n \left[\mu(\eta_0, t_i)-\mu(\eta,t_i)\right]^2 &\geqslant C |\beta-\beta_0|^2.
\end{align*}
Notice now that
\begin{align*}
|u-u_0| &= |\gamma_0^{-1}\beta_0 - \gamma^{-1}\beta| \\
&= |\gamma_0^{-1}| |\beta_0 - \gamma_0\gamma^{-1}\beta| \\
&\leqslant |\gamma_0^{-1}| \left\{|\beta_0 - \beta| + |\beta - \gamma_0\gamma^{-1}\beta|\right\} \\
&\leqslant |\gamma_0^{-1}| \left\{|\beta_0 - \beta| + |\gamma^{-1}\beta|\, |\gamma - \gamma_0|\right\}  \\
&\leqslant |\gamma_0^{-1}| (|\beta_0 - \beta| + |u|\,|\gamma - \gamma_0|) \\
&\leqslant |\gamma_0^{-1}| (1 + \max(|\underline{u}|,|\overline{u}|)) \cdot \max(|\beta_0 - \beta|, |\gamma - \gamma_0| ).
\end{align*}
From here, since \(u\in\interff{\underline{u}}{\overline{u}}\) is bounded, it is straightforward that there exists \(C\in\R_+^*\) such that for any \(n\) large enough
\begin{align*}
n^{-1}\sum_{i=1}^n \left[\mu(\eta_0, t_i)-\mu(\eta,t_i)\right]^2 &\geqslant C |\gamma-\gamma_0|^2, \\
n^{-1}\sum_{i=1}^n \left[\mu(\eta_0, t_i)-\mu(\eta,t_i)\right]^2 &\geqslant C |u-u_0|^2,
\end{align*}
which ends the proof.
\end{proof}

\begin{lemma}\label{lemma.An}
Recall the definition of \(A_{1:n}^*\) given in \eqref{eq.definitionAnstar}. Under Assumptions (A1)--(A4) and conditions \eqref{eq.definitiondn}, as \(n\conv+\infty\)
\begin{align}\label{eq.lemma.An1}
n^{-\frac{1}{2}} A_{1:n}^*(\theta_0) \convloi \loi{N}\left(0, I(\theta_0) \right).
\end{align}
\end{lemma}

\begin{proof}[Proof of Lemma \ref{lemma.An}]
We will show that any linear combination of the coordinates of \(A_{1:n}(\theta_0)\) is asymptotically normal using Lyapounov's Theorem. Let \(\alpha\in\R^3\), \(\|\alpha\|\neq 0\), so that differential calculus allows us to write
\begin{align}
\left<\alpha, A_{1:n}^*(\theta_0)\right>
&= \alpha_1 \cdot \left.\frac{\partial l_{1:n}^*(X_{1:n}|\theta)}{\partial \gamma}\right|_{\theta_0}
 + \alpha_2 \cdot \left.\frac{\partial l_{1:n}^*(X_{1:n}|\theta)}{\partial u}\right|_{\theta_0}
 + \alpha_3 \cdot \left.\frac{\partial l_{1:n}^*(X_{1:n}|\theta)}{\partial \sigma^2}\right|_{\theta_0} \nonumber \\
&= \alpha_1 \cdot \frac{1}{\sigma_0^2}\sum_{i=1}^{n^*} \left[ (t_i - u_0)\one_{\interfo{t_i}{+\infty}}(u_0) \cdot \xi_i \right]
 - \alpha_2 \cdot \frac{\gamma_0}{\sigma_0^2} \sum_{i=1}^{n^*} \left[ \one_{\interfo{t_i}{+\infty}}(u_0) \cdot \xi_i \right] \nonumber\\
&\qquad + \alpha_3 \cdot \frac{1}{2\sigma_0^2} \sum_{i=1}^{n^*} \left[ \frac{1}{\sigma_0^2} \cdot \xi_i^2 - 1 \right] \nonumber \\
&= \sigma_0^{-2} \sum_{i=1}^{n^*} Z_i, \nonumber
\intertext{where we denote, for \(i=1,\ldots,{n^*}\)}
Z_i &= \left[ \Big\{(t_i - u_0)\one_{\interfo{t_i}{+\infty}}(u_0)\cdot \alpha_1 - \gamma_0\one_{\interfo{t_i}{+\infty}}(u_0)\cdot \alpha_2\Big\}\cdot \xi_i + \frac{1}{2}\alpha_3 \cdot \left\{\sigma_0^{-2}\cdot \xi_i^2 - 1\right\} \right].\label{eq.definitionZi}
\end{align}
Since for \(i=1,\ldots,{n^*}\)  \(\esp[\xi_i] = 0\) and \(\esp[\xi_i^2] = \sigma^2\), we deduce that \(\esp\left[Z_i\right] = 0\), and hence that \(\esp\left[\left<\alpha, A_{1:n}^*(\theta_0)\right>\right] = 0\).

Let us now find the expression of \(\var\left<\alpha, A_{1:n}^*(\theta_0)\right>\). Because \(\xi_i\) and \(\xi_j\) are independent when \(i\neq j\), so are \(Z_i\) and \(Z_j\) and we hence write
\begin{align*}
\var\left<\alpha, A_{1:n}^*(\theta_0)\right>
&= \sigma_0^{-4} \sum_{i=1}^{n^*} \var Z_i \\
&= \sigma_0^{-4} \sum_{i=1}^{n^*} \left\{ \Big[(t_i - u_0)\one_{\interfo{t_i}{+\infty}}(u_0)\cdot \alpha_1 - \gamma_0\one_{\interfo{t_i}{+\infty}}(u_0)\cdot \alpha_2\Big]^2\cdot \var \xi_i \right.\\
&\qquad\qquad\qquad \left.+ \frac{1}{4}\alpha_3^2 \cdot \var\left[\sigma_0^{-2}\xi_i^2 - 1\right] \right\}, \\
\intertext{because \(\cov\left[\xi_i, \left\{\sigma_0^{-2}\xi_i^2 - 1\right\} \right] = 0\), and we finally get}
\var\left<\alpha, A_{1:n}^*(\theta_0)\right>
&= \sigma_0^{-4} \sum_{i=1}^{n^*} \left\{ \Big[(t_i - u_0)\one_{\interfo{t_i}{+\infty}}(u_0)\cdot \alpha_1 - \gamma_0\one_{\interfo{t_i}{+\infty}}(u_0)\cdot \alpha_2\Big]^2\cdot \sigma_0^2 + \frac{1}{4}\alpha_3^2 \cdot 2 \right\}.
\end{align*}
We can hence write
\begin{align*}
n^{*-1}\var\left<\alpha, A_{1:n}^*(\theta_0)\right>
&= \sigma_0^{-2} \frac{1}{{n^*}} \sum_{i=1}^{n^*} \Big[(t_i - u_0)\one_{\interfo{t_i}{+\infty}}(u_0)\cdot \alpha_1 - \gamma_0\one_{\interfo{t_i}{+\infty}}(u_0)\cdot \alpha_2\Big]^2 + \frac{1}{2}\sigma_0^{-4}\alpha_3^2 \\
&= \alpha_1^2 \cdot \sigma_0^{-2} \left\{\frac{1}{{n^*}} \sum_{i=1}^{n^*} (t_i - u_0)^2\one_{\interfo{t_i}{+\infty}}(u_0)\right\} \\
&\qquad - 2\alpha_1\alpha_2 \cdot \sigma_0^{-2} \gamma_0 \left\{\frac{1}{{n^*}} \sum_{i=1}^{n^*} (t_i - u_0)\one_{\interfo{t_i}{+\infty}}(u_0)\right\} \\
&\qquad + \alpha_2^2 \cdot \sigma_0^{-2} \gamma_0^2 \left\{\frac{1}{{n^*}} \sum_{i=1}^{n^*} \one_{\interfo{t_i}{+\infty}}(u_0)\right\} + \alpha_3^2 \cdot \frac{1}{2}\sigma_0^{-4} \\
&=\left<\alpha, I_{1:n}(\theta_0) \alpha\right>,
\end{align*}
where we denote
\begin{align}\label{eq.definitionInstar}
I_{1:n}^*(\theta) &=
\left[\begin{array}{ccc}
    \displaystyle \sigma^{-2} \frac{1}{{n^*}} \sum_{i=1}^{n^*} (t_i - u)^2\one_{\interfo{t_i}{+\infty}}(u) & \displaystyle -\sigma^{-2} \gamma \frac{1}{{n^*}} \sum_{i=1}^{n^*} (t_i - u)\one_{\interfo{t_i}{+\infty}}(u) & 0\\
    & \displaystyle \sigma^{-2} \gamma^2 \frac{1}{{n^*}} \sum_{i=1}^{n^*} \one_{\interfo{t_i}{+\infty}}(u)  & 0 \\
    & & \displaystyle \frac{1}{2}\sigma^{-4}
    \end{array}\right].
\end{align}
Remark that, by virtue of Assumption (A1), it is easy to check that for any \(\theta\in\Theta\)
\begin{align}\label{eq.IntendversI}
I_{1:n}^*(\theta)&\conv I(\theta),
\end{align}
and observe that just like \(I(\theta)\), \(I_{1:n}^*(\theta)\) is positive definite, since all its principal minor determinants are positive.

Let us now check that the random variables \(Z_i\) meet Lyapounov's Theorem \cite[see][page 362]{BillingsleyPM} requirements before wrapping up this proof. The random variables \(Z_i\) are independent and trivially \(L^2\). We denote \(V_n^{*2} = \sum_{i=1}^{n^*} \var Z_i\) and claim that Lyapounov's condition holds, that is
\begin{align*}
\exists \delta > 0, \;\sum_{i=1}^{n^*} \esp\left|\frac{Z_i - \esp Z_i}{V_n^*}\right|^{2+\delta} = \po(1).
\end{align*}
Indeed we have (\(\delta=1\))
\begin{align*}
\sum_{i=1}^{n^*}\esp\left|\frac{Z_i - \esp Z_i}{V_n^*}\right|^{3} &= \sum_{i=1}^{n^*} \esp\left|\frac{Z_i}{V_n^*}\right|^{3} \\
&= \frac{{n^*}}{\var^{\frac{3}{2}}\left<\alpha, A_{1:n}^*(\theta_0)\right>} \cdot \frac{1}{{n^*}} \sum_{i=1}^{n^*} \esp\left|Z_i\right|^{3} \\
&= \frac{1}{{n^*}^{\frac{1}{2}}\left<\alpha, I_{1:n}^*(\theta_0) \alpha\right>^{\frac{3}{2}}} \cdot \frac{1}{{n^*}} \sum_{i=1}^{n^*} \esp\left|Z_i\right|^{3}.
\end{align*}
The first term of this last product is \(\go\left(n^{*-\frac{1}{2}}\right)\) thanks to \eqref{eq.IntendversI}, and recalling the definition of \(Z_i\) from \eqref{eq.definitionZi}, there is no difficulty in showing that the last term of the product, namely \(\frac{1}{{n^*}} \sum_{i=1}^{n^*} \esp\left|Z_i\right|^{3}\) converges to a finite limit. Indeed we find, using trivial dominations and Assumption (A1) once again,
\begin{align*}
|Z_i|^3 &= \left| \left\{(t_i - u_0)\one_{\interfo{t_i}{+\infty}}(u_0)\cdot \alpha_1 - \gamma_0\one_{\interfo{t_i}{+\infty}}(u_0)\cdot \alpha_2\right\}\cdot \xi_i + \frac{1}{2}\alpha_3 \cdot \left\{\sigma_0^{-2}\cdot \xi_i^2 - 1\right\} \right|^3 \\
\esp |Z_i|^3 &\leqslant \left(\left|(t_i - u_0)\one_{\interfo{t_i}{+\infty}}(u_0)\cdot \alpha_1 - \gamma_0\one_{\interfo{t_i}{+\infty}}(u_0)\cdot \alpha_2\right| + \left|\frac{1}{2}\alpha_3\right|\right)^3 \\
&\qquad \times \esp \left(\left|\xi_i\right| + \left|\sigma_0^{-2}\cdot \xi_i^2 - 1\right|\right)^3 \\
\frac{1}{n}\sum_{i=1}^n \esp |Z_i|^3 &\leqslant \frac{1}{n} \sum_{i=1}^n \left(\left|(t_i - u_0)\one_{\interfo{t_i}{+\infty}}(u_0)\cdot \alpha_1 - \gamma_0\one_{\interfo{t_i}{+\infty}}(u_0)\cdot \alpha_2\right| + \left|\frac{1}{2}\alpha_3\right|\right)^3 \\
&\qquad \times \esp \left(\left|\xi_i\right| + \left|\sigma_0^{-2}\cdot \xi_i^2 - 1\right|\right)^3 \\
&\leqslant \go(1) \cdot \frac{1}{n} \sum_{i=1}^n \left(\left|(t_i - u_0)\one_{\interfo{t_i}{+\infty}}(u_0)\cdot \alpha_1 - \gamma_0\one_{\interfo{t_i}{+\infty}}(u_0)\cdot \alpha_2\right| + \left|\frac{1}{2}\alpha_3\right|\right)^3 \\
&\leqslant \go(1).
\end{align*}
Lyapounov's Theorem thus applies here and leads to
\begin{align*}
\sum_{i=1}^{n^*}\frac{Z_i-\esp Z_i}{V_n^*} &\convloi \loi{N}(0, 1),
\intertext{i.e. multiplying numerator and denominator by \(\sigma_0^{-2}\) we get}
\frac{\left<\alpha, A_{1:n}^*(\theta_0)\right>}{\var^{\frac{1}{2}}\left<\alpha, A_{1:n}^*(\theta_0)\right>} &\convloi \loi{N}(0, 1),
\intertext{that is}
\frac{\left<\alpha, A_{1:n}^*(\theta_0)\right>}{n^{*\frac{1}{2}}\left<\alpha, I_{1:n}^*(\theta_0) \alpha\right>^{\frac{1}{2}}} &\convloi \loi{N}(0, 1),
\intertext{and because of \eqref{eq.IntendversI} we can also write,}
\frac{\left<\alpha, A_{1:n}^*(\theta_0)\right>}{n^{*\frac{1}{2}}\left<\alpha, I(\theta_0) \alpha\right>^{\frac{1}{2}}} &\convloi \loi{N}(0, 1),
\end{align*}
which, remembering that a.s. \(n^* \sim n\), is equivalent to \eqref{eq.lemma.An1}.
\end{proof}

\begin{lemma}\label{lemma.Bn}
Recall the definition of \(B_{1:n}^*\) given in \eqref{eq.definitionBnstar}. Under Assumptions (A1)--(A4) and conditions \eqref{eq.definitiondn}, as \(n\conv+\infty\),
\begin{align}
&\dfrac{1}{n}B_{1:n}^*(\theta_0) \convps I(\theta_0), \text{ as } n\conv +\infty. \label{eq.lemma.Bn1}\\
&\dfrac{1}{n}B_{1:n}^*(\theta) \convps I(\theta_0), \text{ as } \theta \conv \theta_0 \text{ and  } n\conv +\infty.\label{eq.lemma.Bn2}
\end{align}
where the asymptotic Fisher Information Matrix \(I(\cdot)\) is defined in \eqref{eq.lemmaasympFishInfMatrix}.
\end{lemma}

\begin{proof}[Proof of Lemma \ref{lemma.Bn}]We will prove each claim separately.

\vspace{0.25em}\par\noindent\textsc{Proof of \eqref{eq.lemma.Bn1}.~} Differential calculus provides the following expressions for the coefficients of \(\dfrac{1}{{n^*}}B_{1:n}^*(\theta)\).
\begin{align*}
\left(\dfrac{1}{{n^*}}B_{1:n}^*(\theta)\right)_{11} &= \sigma^{-2}\frac{1}{{n^*}} \sum_{i=1}^{n^*} (t_i - u)^2 \one_{\interfo{t_i}{+\infty}}(u), \\
\left(\dfrac{1}{{n^*}}B_{1:n}^*(\theta)\right)_{12} &= \sigma^{-2}\frac{1}{{n^*}} \sum_{i=1}^{n^*} \left[\xi_i + \gamma_0\cdot(t_i - u_0)\one_{\interfo{t_i}{+\infty}}(u_0)- 2\gamma\cdot(t_i - u)\right] \one_{\interfo{t_i}{+\infty}}(u), \\
\left(\dfrac{1}{{n^*}}B_{1:n}^*(\theta)\right)_{13} &= \sigma^{-4}\frac{1}{{n^*}} \sum_{i=1}^{n^*} \left[\xi_i+\gamma_0\cdot(t_i - u_0)\one_{\interfo{t_i}{+\infty}}(u_0) - \gamma\cdot(t_i - u)\right](t_i - u) \one_{\interfo{t_i}{+\infty}}(u), \\
\left(\dfrac{1}{{n^*}}B_{1:n}^*(\theta)\right)_{22} &= \sigma^{-2}\gamma^2\frac{1}{{n^*}} \sum_{i=1}^{n^*} \one_{\interfo{t_i}{+\infty}}(u), \\
\left(\dfrac{1}{{n^*}}B_{1:n}^*(\theta)\right)_{23} &= -\sigma^{-4}\gamma\frac{1}{{n^*}} \sum_{i=1}^{n^*} \left[\xi_i+\gamma_0\cdot(t_i - u_0)\one_{\interfo{t_i}{+\infty}}(u_0)-\gamma\cdot(t_i - u)\right] \one_{\interfo{t_i}{+\infty}}(u), \\
\left(\dfrac{1}{{n^*}}B_{1:n}^*(\theta)\right)_{33} &= -\frac{1}{2}\sigma^{-4} + \sigma^{-6} \frac{1}{{n^*}} \sum_{i=1}^{n^*} \left[\xi_i+\gamma_0\cdot(t_i - u_0)\one_{\interfo{t_i}{+\infty}}(u_0) - \gamma\cdot(t_i - u)\one_{\interfo{t_i}{+\infty}}(u)\right]^2.
\end{align*}
The convergence we claim is then a direct consequence of Assumption (A1) and the fact that \(n^{*}\sim n\) and, depending on the coefficients, either the Strong Law of Large Numbers or Kolmogorov's criterion. Notice that
\begin{align*}
\frac{1}{n^*}B_{1:n}^*(\theta_0) - I_{1:n}^*(\theta_0) \convps 0,
\end{align*}
where \(I_{1:n}^*\) is defined in \eqref{eq.definitionInstar}.
\vspace{0.25em}\par\noindent\textsc{Proof of \eqref{eq.lemma.Bn2}.~} We will show that in fact, as \(n\conv+\infty\) and \(\theta\conv\theta_0\),
\begin{align*}
C_{1:n}^*(\theta)&=\dfrac{1}{{n^*}}B_{1:n}^*(\theta_0)-\dfrac{1}{{n^*}}B_{1:n}^*(\theta) \convps 0,
\end{align*}
which will end the proof since \(n^{*}\sim n\). We will consider each coefficient of \(C_{1:n}^*(\theta)\) in turn, making use of Assumption (A1) once again and apply repeatedly the Strong Law of Large Numbers and Kolmogorov's criterion as well as Lemma \ref{lemma.convpuissk}, whenever needed.

\begin{align*}
C_{1:n}^*(\theta)_{11}
&= \sigma_0^{-2} \frac{1}{{n^*}} \sum_{i=1}^{n^*} (t_i - u_0)^2 \one_{\interfo{t_i}{+\infty}}(u_0) - \sigma^{-2}\frac{1}{{n^*}} \sum_{i=1}^{n^*} (t_i - u)^2 \one_{\interfo{t_i}{+\infty}}(u) \\
&= \left(\sigma_0^{-2} - \sigma^{-2}\right) \cdot \frac{1}{{n^*}} \sum_{i=1}^{n^*} (t_i - u_0)^2 \one_{\interfo{t_i}{+\infty}}(u_0) \\
&\quad\quad + \sigma^{-2} \cdot \left( \frac{1}{{n^*}} \sum_{i=1}^{n^*} (t_i - u_0)^2 \one_{\interfo{t_i}{+\infty}}(u_0) - \frac{1}{{n^*}} \sum_{i=1}^{n^*} (t_i - u)^2 \one_{\interfo{t_i}{+\infty}}(u) \right) \\
&= \po(1) \cdot \go(1)+\go(1)\cdot \go\left(u - u_0\right) \conv 0.
\intertext{then last equality holding true because of Lemma \ref{lemma.convpuissk}.}
C_{1:n}^*(\theta)_{22} &= \sigma_0^{-2}\gamma_0^2\frac{1}{{n^*}} \sum_{i=1}^{n^*} \one_{\interfo{t_i}{+\infty}}(u_0) - \sigma^{-2}\gamma^2\frac{1}{{n^*}} \sum_{i=1}^{n^*} \one_{\interfo{t_i}{+\infty}}(u) \\
&= \left(\sigma_0^{-2}\gamma_0^2 - \sigma^{-2}\gamma^2\right) \cdot \frac{1}{{n^*}} \sum_{i=1}^{n^*} \one_{\interfo{t_i}{+\infty}}(u_0) \\
&\qquad + \sigma^{-2}\gamma^2 \cdot \left[\frac{1}{{n^*}} \sum_{i=1}^{n^*} \one_{\interfo{t_i}{+\infty}}(u_0) - \frac{1}{{n^*}} \sum_{i=1}^{n^*} \one_{\interfo{t_i}{+\infty}}(u)\right] \\
&= \po(1)\cdot \go(1) + \go(1) \cdot \left[\{F_{n^*}(u_0) - F(u_0)\} + \{F(u_0) - F(u)\} + \{F(u) - F_{n^*}(u)\}\right] \\
&= \po(1) + \go(1) \cdot \left[\po(1) + \po(1) + \po(1)\right] \conv 0,
\intertext{the last equality holding true because of the uniform convergence of \(F_{n^*}\) to \(F\) over any compact subset such as \(\interff{\underline{u}}{\overline{u}}\) (see Assumption (A1), and its Remark 1).}
C_{1:n}^*(\theta)_{33} &= \frac{1}{2}\sigma^{-4} - \sigma^{-6} \frac{1}{{n^*}} \sum_{i=1}^{n^*} [\xi_i +\gamma_0\cdot(t_i - u_0)\one_{\interfo{t_i}{+\infty}}(u_0) - \gamma\cdot(t_i - u)\one_{\interfo{t_i}{+\infty}}(u)]^2 \\
&\qquad- \left(\frac{1}{2}\sigma_0^{-4} - \sigma_0^{-6} \frac{1}{{n^*}} \sum_{i=1}^{n^*} \xi_i^2 \right)\\
&=\frac{1}{2}(\sigma^{-4} - \sigma_0^{-4}) - (\sigma^{-6} - \sigma_0^{-6}) \cdot \frac{1}{{n^*}} \sum_{i=1}^{n^*} \xi_i^2 \\
&\qquad- \sigma^{-6} \frac{1}{{n^*}} \sum_{i=1}^{n^*} \left[\gamma_0\cdot(t_i - u_0)\one_{\interfo{t_i}{+\infty}}(u_0) - \gamma\cdot(t_i - u)\one_{\interfo{t_i}{+\infty}}(u)\right]\xi_i \\
& \qquad - \sigma^{-6} \frac{1}{{n^*}} \sum_{i=1}^{n^*} \left[\gamma_0\cdot(t_i - u_0)\one_{\interfo{t_i}{+\infty}}(u_0) - \gamma\cdot(t_i - u)\one_{\interfo{t_i}{+\infty}}(u)\right]^2 \\
&= \po(1) + \po(1) \cdot \frac{1}{{n^*}} \sum_{i=1}^{n^*} \xi_i^2 + \po(1) + \po(1) \convps 0,
\intertext{where the two last \(\po(1)\) are direct consequences of Lemmas \ref{lemma.convergenceuniformemu} and \ref{lemma.sylwestertheo35}.}
\intertext{Those same Lemmas used together with Lemma \ref{lemma.convpuissk}, the Strong Law of Large Numbers as well as the well-known Cauchy-Schwarz inequality imply that a.s.}
C_{1:n}^*(\theta)_{23} &= \sigma^{-4}\gamma\frac{1}{{n^*}} \sum_{i=1}^{n^*} \left[\xi_i+\gamma_0\cdot(t_i - u_0)\one_{\interfo{t_i}{+\infty}}(u_0)-\gamma\cdot(t_i - u)\right] \one_{\interfo{t_i}{+\infty}}(u) \\
&\qquad -\sigma_0^{-4}\gamma_0\frac{1}{{n^*}} \sum_{i=1}^{n^*} \xi_i \one_{\interfo{t_i}{+\infty}}(u_0) \\
&= \frac{1}{{n^*}} \sum_{i=1}^{n^*} \left[\sigma^{-4}\gamma\one_{\interfo{t_i}{+\infty}}(u)-\sigma_0^{-4}\gamma_0\one_{\interfo{t_i}{+\infty}}(u_0)\right]\xi_i \\
&\qquad + \frac{1}{{n^*}} \sum_{i=1}^{n^*} \left[\gamma_0\cdot(t_i - u_0)\one_{\interfo{t_i}{+\infty}}(u_0)-\gamma\cdot(t_i - u)\right] \sigma^{-4}\gamma\one_{\interfo{t_i}{+\infty}}(u) \\
&= \po(1) + \po(1) \convps 0,
\intertext{and also that a.s.}
C_{1:n}^*(\theta)_{13} &= \sigma_0^{-4}\frac{1}{{n^*}} \sum_{i=1}^{n^*} \left[(t_i - u_0) \one_{\interfo{t_i}{+\infty}}(u_0)\right]\xi_i \\
&\qquad -\sigma^{-4}\frac{1}{{n^*}} \sum_{i=1}^{n^*} \left[\xi_i+\gamma_0\cdot(t_i - u_0)\one_{\interfo{t_i}{+\infty}}(u_0)-\gamma\cdot(t_i - u)\right](t_i - u) \one_{\interfo{t_i}{+\infty}}(u)\\
&=\frac{1}{{n^*}} \sum_{i=1}^{n^*} \left[\sigma_0^{-4}(t_i - u_0) \one_{\interfo{t_i}{+\infty}}(u_0)-\sigma^{-4}(t_i - u) \one_{\interfo{t_i}{+\infty}}(u)\right] \xi_i \\
&\qquad - \sigma^{-4} \frac{1}{{n^*}} \sum_{i=1}^{n^*} \left[\gamma_0\cdot(t_i - u_0)\one_{\interfo{t_i}{+\infty}}(u_0)-\gamma\cdot(t_i - u)\right](t_i - u) \one_{\interfo{t_i}{+\infty}}(u)\\
&= \po(1) + \po(1) \convps 0.
\intertext{and finally that a.s.}
C_{1:n}^*(\theta)_{12} &= \sigma_0^{-2}\frac{1}{{n^*}} \sum_{i=1}^{n^*} \left[\xi_i - \gamma_0\cdot(t_i - u_0)\right] \one_{\interfo{t_i}{+\infty}}(u_0) \\
&\qquad - \sigma^{-2}\frac{1}{{n^*}} \sum_{i=1}^{n^*} \left[\xi_i + \gamma_0\cdot(t_i - u_0)\one_{\interfo{t_i}{+\infty}}(u_0) - 2\gamma\cdot(t_i - u)\right] \one_{\interfo{t_i}{+\infty}}(u) \\
&= \frac{1}{{n^*}} \sum_{i=1}^{n^*} \xi_i \cdot \left[\sigma_0^{-2}\one_{\interfo{t_i}{+\infty}}(u_0)   - \sigma^{-2}\one_{\interfo{t_i}{+\infty}}(u)\right] \\
&\qquad +\frac{1}{{n^*}} \sum_{i=1}^{n^*} \left[-\sigma_0^{-2}\gamma_0\cdot(t_i - u_0) \one_{\interfo{t_i}{+\infty}}(u_0) - \sigma^{-2}(\gamma_0\cdot(t_i - u_0) \one_{\interfo{t_i}{+\infty}}(u_0) \right. \\
&\qquad\qquad\qquad\qquad \left. - 2\gamma\cdot(t_i - u) \one_{\interfo{t_i}{+\infty}}(u))\right] \\
&= \po(1) + \po(1) \convps 0.
\end{align*}
\end{proof}

\begin{proposition}\label{prop.majorationdiffdeL}
Let \(0<\delta\), and let \((\rho_n)_{n\in\N}\) be a positive sequence such that, as \(n\conv+\infty\)
\begin{align}
\rho_n &= \go(1) \label{eq.definitionrhon1}\\
n^{-\frac{1}{2}} (\log n) \cdot \rho_n^{-1} &\conv 0 \label{eq.definitionrhon2}
\end{align}
and denote
\begin{align*}
B^c(\theta_0, \delta \rho_n) = \left\{\theta \in \Theta,\; \|\theta-\theta_0\| \geqslant \delta \rho_n\right\},
\end{align*}
Then, under Assumptions (A1)--(A4), a.s., there exists \(\epsilon> 0\) such that, for any \(n\) large enough
\begin{align}
\sup_{\theta \in B^c(\theta_0, \delta \rho_n)} \frac{1}{n \rho_n^2} [l_{1:n}(X_{1:n}|\theta) - l_{1:n}(X_{1:n}|\widehat{\theta}_n)] &\leqslant -\epsilon. \label{eq.supdiffdeLrhon1} \\
\sup_{\theta \in B^c(\theta_0, \delta \rho_n)} \frac{1}{n \rho_n^2} [l_{1:n}(X_{1:n}|\theta) - l_{1:n}(X_{1:n}|\theta_0)] &\leqslant -\epsilon. \label{eq.supdiffdeLrhon2}
\end{align}
\end{proposition}

\begin{proof}[Proof of Proposition \ref{prop.majorationdiffdeL}]This proposition is to be compared to the regularity condition imposed in \cite{Ghosh} (see their condition (A4) in Chapter 4). The aim of this proposition is to show that our model satisfies to a somewhat stronger version of that condition.

Let \(0<\delta\). Notice first that, similarly to what was done in \eqref{eq.sylwesterps}, we are able to deduce that a.s.
\begin{align}
\frac{2}{n}[l_{1:n}(X_{1:n}|\theta) - l_{1:n}(X_{1:n}|\widehat{\theta}_n)] &\leqslant \frac{2}{n}[l_{1:n}(X_{1:n}|\theta) - l_{1:n}(X_{1:n}|\theta_0)] =: i_n(\theta) \label{eq.majorationdiffdeLparin}.
\end{align}
where \(i_n\) is defined over \(\R \times \interff{\underline{u}}{\overline{u}} \times \R_+^* \supset \Theta \supset B^c(\theta_0, \delta \rho_n)\) by
\begin{align}
i_n(\theta) &= \log\frac{\sigma_0^2}{\sigma^2} +1 + \frac{1}{n\sigma_0^2}\sum_{i=1}^n (\xi_i^2 - \sigma_0^2) - \frac{1}{n\sigma^2}\sum_{i=1}^n [\xi_i + \mu(\eta_0, t_i) - \mu(\eta, t_i)]^2. \label{eq.definitionin1} \\
&= \log\frac{\sigma_0^2}{\sigma^2} +1 - \frac{\sigma_0^2}{\sigma^2} + \frac{1}{n\sigma_0^2}\sum_{i=1}^n (\xi_i^2 - \sigma_0^2) - \frac{1}{n\sigma^2}\sum_{i=1}^n \left\{ [\xi_i + \mu(\eta_0, t_i) - \mu(\eta, t_i)]^2 - \sigma_0^2\right\}. \label{eq.definitionin2}
\end{align}
From \eqref{eq.majorationdiffdeLparin} it is clear that we need only prove \eqref{eq.supdiffdeLrhon2} to end the proof.

The rest of this proof is divided into 6 major steps. Step 1 shows that for a given \(n\) the supremum considered is reached on a point \(\theta_n\). Step 2 and 3 focus on obtaining useful majorations of the supremum. Step 4 is dedicated to proving that the sequence \(\theta_n\) admits an accumulation point (the coordinates of which satisfy to some conditions), while step 5 makes use of this last fact to effectively dominate the supremum. Step 6 wraps up the proof.

\vspace{0.25em}\par\noindent\textsc{Step 1.~} We first show that a.s. for any \(n\) there exists \(\theta_n\in \R \times \interff{\underline{u}}{\overline{u}} \times \R_+^*\) such that \(\|\theta_n-\theta_0\|\geqslant \delta \rho_n\) and
\begin{align}
i_n(\theta_n) = \sup_{\Theta \in B^c(\theta_0, \delta \rho_n)} i_n(\theta) \label{eq.inatteintsonsup}.
\end{align}
Let \(n\in\N\) and let \((\theta_{n,k})_{k\in\N}\) be a sequence of points in \(B^c(\theta_0, \delta \rho_n)\) such that
\begin{align*}
\lim_{k\conv+\infty} i_n(\theta_{n,k}) = \sup_{\Theta \in B^c(\theta_0, \delta \rho_n)} i_n(\theta).
\end{align*}
From \eqref{eq.definitionin1} it is obvious that \(\sigma_{n,k}^2\) is bounded: if it was not, we would be able to extract a subsequence such that \(\sigma_{n,k_j}^2\) would go to \(+\infty\) and thus \(i_n(\theta_{n,k_j})\) would go to \(-\infty\). For the very same reason, \(\gamma_{n,k}\) too is bounded. Recalling that \(u_{n,k}\) is bounded too by definition, we now see that there exists a subsequence \((\theta_{n,k_j})_{j\in\N}\) in \(B^c(\theta_0, \delta \rho_n)\) and a point \(\theta_n\) in \(\overline{B^c(\theta_0, \delta \rho_n})\) (i.e. in \(\R \times \interff{\underline{u}}{\overline{u}} \times \R_+\), and such that \(\|\theta_n-\theta_0\| \geqslant \delta \rho_n\)) such that \((\theta_{n,k_j})_{j\in\N}\xrightarrow[j\conv+\infty]{}\theta_n\).

Finally from \eqref{eq.definitionin1} again it is easy to see that \(\sigma_n^2>0\) for if it was not \(i_n(\theta_{n,k_j})\) would go to \(-\infty\) once again, unless (by continuity of \(\mu\) with regard to \(\eta\)) \(\xi_i + \mu(\eta_0, t_i) - \mu(\eta_n, t_i) = 0\) for all \(i\leqslant n\) which a.s. does not happen.

\vspace{0.25em}\par\noindent\textsc{Step 2.~} From the previous step and the continuity of \(i_n\) with regard to \(\theta\) we are able to write
\begin{align}
\sup_{\Theta \in B^c(\theta_0, \delta \rho_n)} \frac{2}{n}[l_{1:n}(X_{1:n}|\theta) - l_{1:n}(X_{1:n}|\theta_0)] &= i_n(\theta_n) \label{eq.diffdeLetinthetan}.
\end{align}
where \((\theta_n)_{n\in\N}\) is the sequence defined in Step 1. We now derive a convenient majoration of \(i_n(\theta_n)\). Expanding from \eqref{eq.definitionin2} we get
\begin{align*}
i_n(\theta_n) &= \left(\log\frac{\sigma_0^2}{\sigma_n^2} +1 - \frac{\sigma_0^2}{\sigma_n^2}\right) + \frac{1}{n\sigma_0^2}\sum_{i=1}^n (\xi_i^2 - \sigma_0^2) - \frac{1}{n\sigma^2}\sum_{i=1}^n \left\{ [\xi_i + \mu(\eta_0, t_i) - \mu(\eta, t_i)]^2 - \sigma_0^2\right\} \\
&= \left(\log\frac{\sigma_0^2}{\sigma_n^2} +1 - \frac{\sigma_0^2}{\sigma_n^2}\right) + \frac{\sigma_n^2-\sigma_0^2}{n\sigma_0^2\sigma_n^2}\sum_{i=1}^n (\xi_i^2 - \sigma_0^2) - \frac{1}{n\sigma_n^2}\sum_{i=1}^n [\mu(\eta_0, t_i) - \mu(\eta_n, t_i)]^2 \\
&\qquad - \frac{2}{n\sigma_n^2}\sum_{i=1}^n [\mu(\eta_0, t_i) - \mu(\eta_n, t_i)]\xi_i
\end{align*}
Thanks to Lemma \ref{lemma.majorationetamoinseta0}, we know that there exists  \(C_1\in \R_+^*\) such that
\begin{align*}
i_n(\theta_n)&\leqslant \left(\log\frac{\sigma_0^2}{\sigma_n^2} +1 - \frac{\sigma_0^2}{\sigma_n^2}\right) + \frac{\sigma_n^2-\sigma_0^2}{n\sigma_0^2\sigma_n^2}\sum_{i=1}^n (\xi_i^2 - \sigma_0^2) \\
&\qquad - \frac{1}{\sigma_n^2}C_1\|\eta_n-\eta_0\|^2 - \frac{2}{n\sigma_n^2}\sum_{i=1}^n [\mu(\eta_0, t_i) - \mu(\eta_n, t_i)]\xi_i
\end{align*}
From there, the Law of the Iterated Logarithm and a factorisation of the last term together with Corollary \ref{lemma.suiteiidOlogn} lead to:
\begin{align*}
i_n(\theta_n) &\leqslant \left(\log\frac{\sigma_0^2}{\sigma_n^2} +1 - \frac{\sigma_0^2}{\sigma_n^2}\right) + \frac{1}{\sigma_n^2}|\sigma_n^2-\sigma_0^2|R_{1,n} - \frac{1}{\sigma_n^2}C_1\|\eta_n-\eta_0\|^2 \\
&\qquad + \frac{1}{n\sigma_n^2}\left(\sum_{i=1}^n [\mu(\eta_0, t_i) - \mu(\eta_n, t_i)]^2\right)^{\frac{1}{2}} R_{2,n}
\end{align*}
where a.s. \(R_{1,n} = \go\left(n^{-\frac{1}{2}}(\log\log n)^\frac{1}{2}\right)\) and \(R_{2,n} = \go\left(\log n\right)\).
Lemma \ref{lemma.convergenceuniformemu} ensures there exists \(C_2\in \R_+^*\) such that
\begin{align*}
i_n(\theta_n) &\leqslant \left(\log\frac{\sigma_0^2}{\sigma_n^2} +1 - \frac{\sigma_0^2}{\sigma_n^2}\right) + \frac{1}{\sigma_n^2}|\sigma_n^2-\sigma_0^2|R_{1,n} - \frac{1}{\sigma_n^2}C_1\|\eta_n-\eta_0\|^2 + \frac{1}{n\sigma_n^2} C_2 n^{\frac{1}{2}}\|\eta_n-\eta_0\| R_{2,n}
\end{align*}
We thus deduce that there exists \(C\in\R_+^*\) such that:
\begin{align}
i_n(\theta_n) &\leqslant \left(\log\frac{\sigma_0^2}{\sigma_n^2} +1 - \frac{\sigma_0^2}{\sigma_n^2}\right) - \frac{1}{\sigma_n^2}C\|\eta_n-\eta_0\|^2 + \frac{1}{\sigma_n^2} \|\theta_n-\theta_0\| R_n \label{ineq.majorationin}
\end{align}
where a.s. \(R_n = \go\left(n^{-\frac{1}{2}}\log n\right)\). Notice in particular that, due to \eqref{eq.definitionrhon2}, \(R_n = \po(\rho_n)\).

\vspace{0.25em}\par\noindent\textsc{Step 3.~} We obtain two majorations, \eqref{eq.inpolar2} and \eqref{ineq.inpolar1}, that we will make use of in the coming steps. Using a conversion of \(\theta = (\gamma, u, \sigma^2)\) into the spherical coordinate system we write \(\theta_n\) as
\begin{align*}
\theta_n = (r_n\cos\psi_n\cos\phi_n, r_n\sin\psi_n\cos\phi_n, r_n\sin\phi_n),
\end{align*}
where
\begin{align*}
(r_n,\psi_n,\phi_n) &\in \R_+^* \times \interff{0}{2\pi} \times \interoo{0}{\pi},
\end{align*}
and deduce from \eqref{ineq.majorationin} that
\begin{align}
i_n(\theta_n) &\leqslant \left(\log\frac{\sigma_0^2}{r_n\sin\phi_n} + 1 - \frac{\sigma_0^2}{r_n\sin\phi_n}\right) - C r_n\frac{\cos^2\phi_n}{\sin\phi_n} + \frac{1}{\sin\phi_n} R_n \label{eq.inpolar1} \\
&\leqslant \left(\log\frac{\sigma_0^2}{r_n\sin\phi_n} + 1 - \frac{\sigma_0^2}{r_n\sin\phi_n}\right)  + \frac{1}{\sin\phi_n} \left[R_n - C r_n \cos^2\phi_n \right]. \label{eq.inpolar2}
\end{align}
From \eqref{eq.inpolar1} 
we also get the following majoration
\begin{align}
i_n(\theta_n) &\leqslant \left(\log\frac{\sigma_0^2}{r_n\sin\phi_n} + 1 - \frac{\sigma_0^2}{r_n\sin\phi_n}\right) + \frac{1}{\sin\phi_n} R_n. \label{ineq.inpolar1} 
\end{align}

\vspace{0.25em}\par\noindent\textsc{Step 4.~} We show that the sequence \((\theta_n)_{n\in\N}\) we built, converges to a finite limit \(\theta_\infty\) (extracting a subsequence if necessary). Extracting a subsequence if necessary, we can assume that \((\psi_n,\phi_n)\conv(\psi_\infty,\phi_\infty)\in\interff{0}{2\pi}\times\interff{0}{\pi}\). We consider the two following mutually exclusive situations.
\vspace{0.25em}\par\noindent\textsc{Situation A: \(\phi_\infty = 0 \mod \pi\).~} In this situation, there exists \(\epsilon > 0\) such that for any \(n\) large enough,
\begin{align*}
\left[R_n - C r_n \cos^2\phi_n \right] &= \left(\frac{R_n}{r_n} - C\cos^2\phi_n\right) r_n \\
&\leqslant - \epsilon r_n,
\end{align*}
because a.s. \(R_n = \po(r_n)\) (since \(R_n = \po(\rho_n)\) and \(r_n \leqslant \rho_n\)). Used together with \eqref{eq.inpolar2}, this leads to
\begin{align*}
i_n(\theta_n) &\leqslant \left(\frac{\sigma_0^2}{r_n\sin\phi_n} - 1 - \log\frac{\sigma_0^2}{r_n\sin\phi_n}\right) - \epsilon \frac{r_n}{\sin\phi_n},
\end{align*}
for any \(n\) large enough and hence \(i_n(\theta_n)\conv-\infty\) whether \(r_n\) goes to zero or not.
\vspace{0.25em}\par\noindent\textsc{Situation B: \(\phi_\infty \neq 0 \mod \pi\).~} In this situation, from \eqref{ineq.inpolar1}, we see that \(r_n\conv 0\) and  \(r_n\conv+\infty\) both lead to \(i_n(\theta_n)\conv-\infty\).

Observing that \(i_n(\theta)\) converges a.s. to a finite value for any \(\theta\in\Theta\) as \(n\conv+\infty\), we see that \(\lim_{n\conv\infty} i_n(\theta_n) = -\infty\) is not possible by construction of the sequence \(\theta_n\), and deduce that, extracting a subsequence if necessary, there exists
\begin{align*}
(r_\infty, \psi_\infty, \phi_\infty) \in \R_+^* \times \interff{0}{2\pi} \times \interoo{0}{\pi},
\end{align*}
such that \(\lim_{n\conv+\infty}\theta_n = \theta_\infty\). Notice that in particular, \(\sigma_\infty^2 > 0\).

\vspace{0.25em}\par\noindent\textsc{Step 5.~} We will now end the proof by showing that there exists \(\epsilon > 0\) such that for any \(n\) large enough
\begin{align}
i_n(\theta_n) &\leqslant -\epsilon \rho_n^2. \label{ineq.minorationunsurdn2in}
\end{align}
We consider the two following mutually exclusive situations.
\vspace{0.25em}\par\noindent\textsc{Situation A: \(\sigma_\infty^2 \neq \sigma_0^2\).~} In this situation, from \eqref{ineq.majorationin} we get
\begin{align*}
i_n(\theta_n) &\leqslant \left(\log\frac{\sigma_0^2}{\sigma_n^2} + 1 - \frac{\sigma_0^2}{\sigma_n^2}\right) + \frac{1}{\sigma_n^2} \left\|\theta_n-\theta_0\right\| R_n
\end{align*}
and the right-hand side converges to
\begin{align*}
\left(\log\frac{\sigma_0^2}{\sigma_\infty^2} + 1 - \frac{\sigma_0^2}{\sigma_\infty^2}\right) &< 0.
\end{align*}
There hence exists \(\epsilon > 0\) such that for any \(n\) large enough
\begin{align*}
i_n(\theta_n) &\leqslant -\epsilon.
\end{align*}
Since \(\rho_n = \go(1)\) by \eqref{eq.definitionrhon1}, \eqref{ineq.minorationunsurdn2in} is a direct consequence of this.
\vspace{0.25em}\par\noindent\textsc{Situation B: \(\sigma_\infty^2 = \sigma_0^2\).~} In this situation, recalling that for any \(x> 0\)
\begin{align*}
\log x + 1 - x \leqslant -\frac{(x-1)^2}{2} +\frac{(x-1)^3}{3},
\end{align*}
we deduce from \eqref{ineq.majorationin} that for any \(n\) large enough
\begin{align*}
i_n(\theta_n) &\leqslant -\frac{1}{2}\left(\frac{\sigma_0^2}{\sigma_n^2}-1\right)^2 + \frac{1}{3}\left(\frac{\sigma_0^2}{\sigma_n^2}-1\right)^3 - \frac{1}{\sigma_n^2} C \|\eta_n-\eta_0\|^2 + \frac{1}{\sigma_n^2} \left\|\theta_n-\theta_0\right\| R_n \\
&\leqslant \left(\frac{\sigma_0^2}{\sigma_n^2}-1\right)^2 \left[\frac{1}{3}\left(\frac{\sigma_0^2}{\sigma_n^2}-1\right)-\frac{1}{2}\right] - \frac{1}{\sigma_n^2} C \|\eta_n-\eta_0\|^2 + \frac{1}{\sigma_n^2} \left\|\theta_n-\theta_0\right\| R_n \\
&\leqslant -\frac{1}{4}\left(\frac{\sigma_0^2}{\sigma_n^2}-1\right)^2 - \frac{1}{\sigma_n^2} C \|\eta_n-\eta_0\|^2 + \frac{1}{\sigma_n^2} \left\|\theta_n-\theta_0\right\| R_n \\
&\leqslant \frac{1}{\sigma_n^2} \left\{-c\left[(\sigma_0^2 -\sigma_n^2)^2 - \|\eta_n-\eta_0\|^2\right] + \|\theta_n-\theta_0\| R_n\right\}
\end{align*}
where \(c = \min(1/4, C) > 0\). It follows that for any \(n\) large enough
\begin{align*}
i_n(\theta_n)&\leqslant \frac{1}{\sigma_n^2} \left(-c \|\theta_n-\theta_0\|^2 + \|\theta_n-\theta_0\| R_n\right) \\
&\leqslant \frac{1}{\sigma_n^2} \|\theta_n-\theta_0\| \left(R_n-c\|\theta_n-\theta_0\|\right).
\end{align*}
Thus, for any \(n\) large enough
\begin{align*}
i_n(\theta_n) &\leqslant \frac{1}{\sigma_n^2} \frac{\left\|\theta_n-\theta_0\right\|}{\rho_n} \left(\frac{R_n}{\rho_n} - c\frac{\left\|\theta_n-\theta_0\right\|}{\rho_n}\right) \rho_n^2.
\end{align*}
Recalling that
\begin{align*}
\left\|\theta_n-\theta_0\right\| &\geqslant \delta \rho_n, \\
R_n &= \po(\rho_n), \\
\sigma_n^2&\conv\sigma_\infty^2 > 0
\end{align*}
we obtain for any \(n\) large enough,
\begin{align*}
i_n(\theta_n) &\leqslant \frac{1}{\sigma_n^2} \frac{\left\|\theta_n-\theta_0\right\|}{\rho_n} \left(- c\frac{\delta}{2}\right) \rho_n^2
\leqslant -\frac{c\delta^2}{2\sigma_n^2}\rho_n^2
\leqslant -\frac{c\delta^2}{3\sigma_\infty^2}\rho_n^2.
\end{align*}
Hence \eqref{ineq.minorationunsurdn2in} holds in this situation too: it suffices to take \(\epsilon = \dfrac{c\delta^2}{3\sigma_\infty^2}\).

We just proved that \eqref{ineq.minorationunsurdn2in} holds in both cases considered.
\vspace{0.25em}\par\noindent\textsc{Step 6.~} \eqref{eq.supdiffdeLrhon2} is a consequence of \eqref{eq.diffdeLetinthetan} and \eqref{ineq.minorationunsurdn2in}.

\end{proof}

\begin{lemma}\label{lemma.dominationloglikethirdderiv}
Let \(0<\delta<1\)
then under Assumptions (A1)--(A4) and conditions \eqref{eq.definitiondn}, a.s. there exists a constant \(C \in R_+^*\) such that for any \(n\) large enough and for any \(1\leqslant i_1,i_2,i_3 \leqslant 3\)
\begin{align}
\left|\frac{1}{n}\frac{\partial^3 l_{1:n}^*(X_{1:n}|\theta)}{\partial_{i_1}\partial_{i_2}\partial_{i_3}}\right| \leqslant C \label{eq.dominationloglikethirdderiv}
\end{align}
for any \(\theta\in B(\theta_0, \delta d_n)\).
\end{lemma}
\begin{proof}[Proof of Lemma \ref{lemma.dominationloglikethirdderiv}]
Let \(0<\delta<1\). We will prove \eqref{eq.dominationloglikethirdderiv} stands true for any \(1\leqslant i_1,i_2,i_3 \leqslant 3\). First notice that for \(n\) large enough, \(\theta\mapsto l_{1:n}^*(X_{1:n}|\theta)\) is indeed infinitely continuously differentiable over \(B(\theta_0, \delta d_n)\) by definition of the pseudo-problem. Any \(\theta\) subsequently considered within this proof is assumed to belong to \(B(\theta_0, \delta d_n)\). Any convergence subsequently mentioned within this proof is uniform in \(\theta\) for \(\theta\in B(\theta_0, \delta d_n)\) for any \(n\) large enough thanks to Theorem \ref{theo.Polya} and Lemma \ref{lemma.sylwestertheo35}.

\vspace{0.25em}\par\noindent\textsc{Proof of \eqref{eq.dominationloglikethirdderiv} for \(\beta=(3,0,0)\).~}
 \begin{align*}
 \frac{1}{n}\frac{\partial^3 l_{1:n}^*(X_{1:n}|\theta)}{(\partial \gamma)^3} = 0.
 \end{align*}
\vspace{0.25em}\par\noindent\textsc{Proof of \eqref{eq.dominationloglikethirdderiv} for \(\beta=(2,1,0)\).~}
 \begin{align*}
 \left|\frac{1}{n}\frac{\partial^3 l_{1:n}^*(X_{1:n}|\theta)}{(\partial \gamma)^2 \partial u}\right| &= \frac{2}{\sigma^2}\left|\frac{1}{n}\sum_{i=1}^{n^*} (t_i -u)\one_{\interoo{t_i}{+\infty}}(u)\right| \convn \frac{2}{\sigma^4} \left|\int_{\underline{u}}^{u} (t-u)f(t) \ud t \right| \leqslant \frac{2}{\sigma^2}|\overline{u}-\underline{u}|.
 \end{align*}
\vspace{0.25em}\par\noindent\textsc{Proof of \eqref{eq.dominationloglikethirdderiv} for \(\beta=(2,0,1)\).~}
 \begin{align*}
 \left|\frac{1}{n}\frac{\partial^3 l_{1:n}^*(X_{1:n}|\theta)}{(\partial \gamma)^2 \partial \sigma^2}\right| &= \frac{1}{\sigma^4}\left|\frac{1}{n}\sum_{i=1}^{n^*} (t_i -u)^2\one_{\interoo{t_i}{+\infty}}(u)\right| \convn \frac{1}{\sigma^4} \left|\int_{\underline{u}}^{u} (t-u)^2f(t) \ud t \right| \leqslant \frac{1}{\sigma^4}|\overline{u}-\underline{u}|^2.
 \end{align*}
\vspace{0.25em}\par\noindent\textsc{Proof of \eqref{eq.dominationloglikethirdderiv} for \(\beta=(1,2,0)\).~}
 \begin{align*}
 \left|\frac{1}{n}\frac{\partial^3 l_{1:n}^*(X_{1:n}|\theta)}{\partial \gamma\cdot (\partial u)^2}\right| &= \frac{2}{\sigma^2}\left|\gamma\frac{1}{n}\sum_{i=1}^{n^*} \one_{\interoo{t_i}{+\infty}}(u)\right| \convn \frac{2}{\sigma^2}\left|\gamma\int_{\underline{u}}^{u} f(t) \ud t \right| \leqslant \frac{2}{\sigma^2}|\gamma|.
 \end{align*}
\vspace{0.25em}\par\noindent\textsc{Proof of \eqref{eq.dominationloglikethirdderiv} for \(\beta=(1,1,1)\).~}
 \begin{align*}
 \left|\frac{1}{n}\frac{\partial^3 l_{1:n}^*(X_{1:n}|\theta)}{\partial \gamma \partial u \partial \sigma^2}\right| &= \frac{1}{\sigma^4}\left|\frac{1}{n}\sum_{i=1}^{n^*} (X_i-2\gamma\cdot(t_i-u))\one_{\interoo{t_i}{+\infty}}(u)\right| \\
&= \frac{1}{\sigma^4}\left|\frac{1}{n}\sum_{i=1}^{n^*} \left[\xi_i + \gamma_0\cdot(t_i-u_0)\one_{\interoo{t_i}{+\infty}}(u_0) -2\gamma\cdot(t_i-u)\right]\one_{\interoo{t_i}{+\infty}}(u)\right| \\
&\convnps \frac{1}{\sigma^4} \left|\int_{\underline{u}}^{\min(u,u_0)} \gamma_0\cdot(t-u_0) f(t) \ud t - 2 \int_{\underline{u}}^{u} \gamma\cdot(t-u) f(t) \ud t \right|
 \end{align*}
And this limit is bounded by \(\frac{3}{\sigma^4}|\overline{u}-\underline{u}|(|\gamma|+|\gamma_0|)\).

\vspace{0.25em}\par\noindent\textsc{Proof of \eqref{eq.dominationloglikethirdderiv} for \(\beta=(1,0,2)\).~}
 \begin{align*}
 \left|\frac{1}{n}\frac{\partial^3 l_{1:n}^*(X_{1:n}|\theta)}{\partial \gamma\cdot (\partial \sigma^2)^2}\right| &= \frac{2}{\sigma^6}\left|\frac{1}{n}\sum_{i=1}^{n^*} \left[x_i-\gamma\cdot(t_i-u)\one_{\interoo{t_i}{+\infty}}(u)\right](t_i-u)\one_{\interoo{t_i}{+\infty}}(u)\right| \\
&= \frac{2}{\sigma^6}\left|\frac{1}{n}\sum_{i=1}^{n^*} \left[\xi_i + \gamma_0\cdot(t_i-u_0)\one_{\interoo{t_i}{+\infty}}(u_0) -\gamma\cdot(t_i-u)\right](t_i-u)\one_{\interoo{t_i}{+\infty}}(u)\right| \\
&\convnps \frac{2}{\sigma^6} \left|\int_{\underline{u}}^{\min(u,u_0)} \gamma_0\cdot(t-u_0)(t-u) f(t) \ud t - \int_{\underline{u}}^{u} \gamma\cdot(t-u)^2 f(t) \ud t \right|
 \end{align*}
And this limit is bounded by \(\frac{4}{\sigma^6}|\overline{u}-\underline{u}|^2(|\gamma|+|\gamma_0|)\).

\vspace{0.25em}\par\noindent\textsc{Proof of \eqref{eq.dominationloglikethirdderiv} for \(\beta=(0,3,0)\).~}
 \begin{align*}
 \left|\frac{1}{n}\frac{\partial^3 l_{1:n}^*(X_{1:n}|\theta)}{(\partial u)^3}\right| &=0.
 \end{align*}
\vspace{0.25em}\par\noindent\textsc{Proof of \eqref{eq.dominationloglikethirdderiv} for \(\beta=(0,2,1)\).~}
 \begin{align*}
 \left|\frac{1}{n}\frac{\partial^3 l_{1:n}^*(X_{1:n}|\theta)}{(\partial u)^2\partial \sigma^2}\right| &= \frac{1}{\sigma^4}\gamma^2\left|\frac{1}{n}\sum_{i=1}^{n^*} \one_{\interoo{t_i}{+\infty}}(u)\right| \convn \frac{1}{\sigma^4} \gamma^2 \left|\int_{\underline{u}}^{u} f(t) \ud t \right| \leqslant \frac{1}{\sigma^4}\gamma^2.
 \end{align*}
\vspace{0.25em}\par\noindent\textsc{Proof of \eqref{eq.dominationloglikethirdderiv} for \(\beta=(0,1,2)\).~}
 \begin{align*}
 \left|\frac{1}{n}\frac{\partial^3 l_{1:n}^*(X_{1:n}|\theta)}{\partial u (\partial \sigma^2)^2}\right| &= \frac{2}{\sigma^6}\left|\frac{1}{n}\sum_{i=1}^{n^*} \left[x_i-\gamma\cdot(t_i-u)\one_{\interoo{t_i}{+\infty}}(u)\right]\gamma\one_{\interoo{t_i}{+\infty}}(u)\right| \\
&= \frac{2}{\sigma^6}\Bigg|\frac{1}{n}\sum_{i=1}^{n^*} \left[\xi_i + \gamma_0\cdot(t_i-u_0)\one_{\interoo{t_i}{+\infty}}(u_0)  -\gamma\cdot(t_i-u)\right]\gamma\one_{\interoo{t_i}{+\infty}}(u)\Bigg| \\
&\convnps \frac{2}{\sigma^6} \left|\int_{\underline{u}}^{u_0} \gamma\gamma_0\cdot(t-u_0) f(t) \ud t - \int_{\underline{u}}^{u} \gamma^2(t-u) f(t) \ud t \right|
 \end{align*}
And this limit is bounded by \(\frac{2}{\sigma^6}|\overline{u}-\underline{u}|(|\gamma^2|+|\gamma_0\gamma|)\).

\vspace{0.25em}\par\noindent\textsc{Proof of \eqref{eq.dominationloglikethirdderiv} for \(\beta=(0,0,3)\).~}
 \begin{align*}
 \left|\frac{1}{n}\frac{\partial^3 l_{1:n}^*(X_{1:n}|\theta)}{(\partial \sigma^2)^3}\right| &= \frac{1}{\sigma^6}\left|-1 + \frac{3}{\sigma^2}\frac{1}{n}\sum_{i=1}^{n^*} \left[x_i-\gamma\cdot(t_i-u)\one_{\interoo{t_i}{+\infty}}(u)\right]^2 \right| \\
&= \frac{1}{\sigma^6}\left|-1 + \frac{3}{\sigma^2}\frac{1}{n}\sum_{i=1}^{n^*} \left[\xi_i+\gamma\cdot(t_i-u_0)\one_{\interoo{t_i}{+\infty}}(u_0)-\gamma\cdot(t_i-u)\one_{\interoo{t_i}{+\infty}}(u)\right]^2 \right| \\
&\convnps \frac{1}{\sigma^8}\left|-\sigma^2 + 3\left(\sigma_0^2 + \int_{\underline{u}}^{u_0} \gamma_0^2(t-u_0)^2 f(t) \ud t \right.\right.\\
&\qquad \qquad \left.\left. -2 \int_{\underline{u}}^{\min(u,u_0)} \gamma\gamma_0\cdot(t-u_0)(t-u) f(t) \ud t + \int_{\underline{u}}^{u} \gamma^2(t-u)^2 f(t) \ud t \right) \right|
 \end{align*}
And this limit is bounded by \(\frac{1}{\sigma^8}\left[3\sigma_0^2+\sigma^2+(|\gamma|+|\gamma_0|)^2(\overline{u}-\underline{u})^2\right]\).

\eqref{eq.dominationloglikethirdderiv} is thus a direct consequence of both the uniform convergences mentioned above and the trivial majoration of all the limits involved by a fixed constant \(C\) for any \(n\) large enough.
\end{proof}

\section*{Acknowledgements}
The authors would like to thank the referees and associate editor for their constructive comments.

\bibliographystyle{apalike}

\begin{thebibliography}{}

\end{thebibliography}


\begin{thebibliography}{}

\bibitem[Billingsley, 1995]{BillingsleyPM}
Billingsley, P. (1995).
\newblock {\em Probability and Measure}.
\newblock Wiley, 3rd edition.

\bibitem[Billingsley, 1999]{BillingsleyCPM}
Billingsley, P. (1999).
\newblock {\em Convergence of Probability Measures}.
\newblock Wiley, 2nd edition.

\bibitem[Breiman, 1992]{Breiman}
Breiman, L. (1992).
\newblock {\em Probability}.
\newblock SIAM.

\bibitem[Bruhns et~al., 2005]{Bruhns}
Bruhns, A., Deurveilher, G., and Roy, J. (2005).
\newblock A non-linear regression model for mid-term load forecasting and
  improvements in seasonnality.
\newblock {\em Proceedings of the 15th Power Systems Computation Conference
  2005, Liege Belgium}.

\bibitem[Dacunha-Castelle, 1978]{Dacunha}
Dacunha-Castelle, D. (1978).
\newblock Vitesse de convergence pour certains probl\`emes statistiques.
\newblock In {\em \'{E}cole d'\'{E}t\'e de {P}robabilit\'es de {S}aint-{F}lour,
  {VII} ({S}aint-{F}lour, 1977)}, volume 678 of {\em Lecture Notes in Math.},
  pages 1--172. Springer, Berlin.

\bibitem[Feder, 1975]{Feder}
Feder, P.~I. (1975).
\newblock On asymptotic distribution theory in segmented regression problems --
  identified case.
\newblock {\em The Annals of Statistics}, 3(1):49--83.

\bibitem[Ghosal and Samanta, 1995]{GhosalSamanta}
Ghosal, S. and Samanta, T. (1995).
\newblock Asymptotic behaviour of {B}ayes estimates and posterior distributions
  in multiparameter nonregular cases.
\newblock {\em Math. Methods Statist.}, 4(4):361--388.

\bibitem[Ghosh et~al., 2006]{Ghosh}
Ghosh, J.~K., Delampady, M., and Samanta, T. (2006).
\newblock {\em An Introduction to Bayesian Analysis, Theory and Methods}.
\newblock Springer.

\bibitem[Ghosh et~al., 1994]{GhoshGhosalSamanta}
Ghosh, J.~K., Ghosal, S., and Samanta, T. (1994).
\newblock Stability and convergence of the posterior in non-regular problems.
\newblock In {\em Statistical decision theory and related topics, {V} ({W}est
  {L}afayette, {IN}, 1992)}, pages 183--199. Springer, New York.

\bibitem[Ghosh and Ramamoorthi, 2003]{GhoshBNP}
Ghosh, J.~K. and Ramamoorthi, R.~V. (2003).
\newblock {\em Bayesian Nonparametrics}.
\newblock Springer.

\bibitem[Ibragimov and Has'minskii, 1981]{Ibragimov}
Ibragimov, I. and Has'minskii, R. (1981).
\newblock {\em Statistical Estimation Asymptotic Theory}.
\newblock Springer.

\bibitem[Launay et~al., 2012]{Launay1}
Launay, T., Philippe, A., and Lamarche, S. (2012).
\newblock Construction of an informative hierarchical prior distribution.
  application to electricity load forecasting.
\newblock {\em Preprint}.
\newblock arXiv:1109.4533.

\bibitem[Lehmann, 2004]{Lehmann}
Lehmann, E.~L. (2004).
\newblock {\em Elements of large-sample theory}.
\newblock Springer Texts in Statistics. Springer-Verlag, New York.

\bibitem[Lo\`{e}ve, 1991]{Loeve}
Lo\`{e}ve, M. (1991).
\newblock {\em Proability Theory I}.
\newblock Springer, 4th edition.

\bibitem[Polya and Szeg\"{o}, 2004]{Polya}
Polya, G. and Szeg\"{o}, G. (2004).
\newblock {\em Problems and Theorems in Analysis I}.
\newblock Springer.

\bibitem[Sareen, 2003]{Sareen}
Sareen, S. (2003).
\newblock Reference bayesian inference in non-regular models.
\newblock {\em Journal of Econometrics}, 113:265--288.

\bibitem[Sylwester, 1965]{Sylwester}
Sylwester, D.~L. (1965).
\newblock On maximum likelihood estimation for two-phase linear regression.
\newblock Technical Report~11, Department of Statistics, Stanford University.

\end{thebibliography}

\end{document}